\documentclass[letterpaper,12pt]{amsart}

\usepackage{amssymb,amsmath,amsthm}
\usepackage{mathtools}
\usepackage[divide={1in,*,1in}]{geometry}
\usepackage{enumitem}
\usepackage{pstricks,pst-node}
\usepackage[cmtip]{xypic}
\usepackage{centerpict}
\usepackage{mywrap}
\usepackage{xspace}
\usepackage[%
	pdfauthor={Jozef H. Przytycki,Krzysztof K. Putyra},%
	pdftitle={The degenerate distributive complex is degenerate},colorlinks%
]{hyperref}
\usepackage[savepos]{zref}

\let\cites\cite

\newtheorem{theorem}{Theorem}[section]
\newtheorem{lemma}[theorem]{Lemma}
\newtheorem{proposition}[theorem]{Proposition}
\newtheorem{corollary}[theorem]{Corollary}

\theoremstyle{definition}

\newtheorem{definition}[theorem]{Definition}
\newtheorem{example}[theorem]{Example}
\newtheorem{remark}[theorem]{Remark}

\newcount\repeatedID
\def\repeattheorem#1#2{%
	\theoremstyle{plain}
	\newtheorem*{RepeatedTheorem\the\repeatedID}{#1~\ref{#2}}%
	\begin{RepeatedTheorem\the\repeatedID}%
}
\def\endrepeattheorem{%
	\end{RepeatedTheorem\the\repeatedID}%
	\global\advance\repeatedID 1\relax
}

\definecolor{internalLink}{rgb}{0.5,0,0}
\definecolor{citeLink}{rgb}{0,0.5,0}
\definecolor{urlLink}{rgb}{0,0,0.5}

\hypersetup{linkcolor=internalLink}
\hypersetup{citecolor=citeLink}
\hypersetup{urlcolor=urlLink}

\DeclareMathOperator{\id}{id}		
\DeclareMathOperator{\rk}{rk}		

\def\CmodifierReset{\let\Cmodifier\relax}
\def\CmodifierSet#1{\def\Cmodifier##1{#1##1\CmodifierReset}}
\CmodifierReset

\def\Caugm{\CmodifierSet\widetilde}		
\def\Chat {\CmodifierSet\widehat}		

\def\CN{\Cmodifier{C}^{N\!}}			
\def\CD{\Cmodifier{C}^{D}}			
\def\CL{\Cmodifier{C}^{L}}			

\def\HN{\Cmodifier{H}^{N\!}}			
\def\HD{\Cmodifier{H}^{D}}			

\let\diff\partial		
\def\F{\mathcal{F}}		

\makeatletter
\def\gr{\@ifnextchar_{\@firstoftwo\gr@sub}{\gr@nosub}}
\def\gr@sub#1{\mathrm{gr}_{#1}\mskip\thinmuskip}
\def\gr@nosub{\mathrm{gr}\mskip\thinmuskip}
\makeatother

\newbox\seqbox
\newlength\seqdepth
\def\seq#1{\mathchoice
	{\seqstyled\textfont\displaystyle{#1}}%
	{\seqstyled\textfont\textstyle{#1}}%
	{\seqstyled\scriptfont\scriptstyle{#1}}%
	{\seqstyled\scriptscriptfont\scriptscriptstyle{#1}}%
}
\def\seqstyled#1#2#3{%
	\setbox\seqbox=\hbox{$#2#3$}%
	\seqdepth=\the\fontdimen16#12\relax
	\mathrlap{\rule[-1.25\seqdepth]{\the\wd\seqbox}{\the\fontdimen8#13}}#3%
}

\DeclareMathOperator{\opn}{\star}		

\let\ltriv\dashv		
\let\rtriv\vdash		

\def\totalface{compound face}		
\def\totalaction{compound action}	

\newif\ifpictup
\pictuptrue
\def\upordown#1#2{\ifpictup#1\else#2\fi}

\def\ungrade{\mathcal E}
\def\bisom{\mathcal E}

\makeatletter

\def\arXiv{\@ifstar\arXiv@@\arXiv@}
\def\arXiv@#1{\href{http://front.math.ucdavis.edu/#1}{arXiv:#1}.}
\def\arXiv@@#1#2{\href{http://front.math.ucdavis.edu/#2}{arXiv:#1}.}

\makeatother

\dimendef\diagcouponXa=20
\dimendef\diagcouponXb=21
\dimendef\diagcouponYa=22
\dimendef\diagcouponYb=23

\newrgbcolor{bandcolor}{0.7 0 0.7}

\newpsobject{mycoupon}{psframe}{framearc=0.5,fillstyle=solid,fillcolor=white}
\newpsobject{bandline}{psline}{linewidth=0.5\psxunit,linecolor=bandcolor}
\newpsobject{bandcustom}{pscustom}{linestyle=none,fillstyle=solid,fillcolor=bandcolor}
\newpsobject{bandframe}{psframe}{linestyle=none,,fillstyle=solid,fillcolor=bandcolor}

\makeatletter
\def\diagcoupon(#1,#2)(#3,#4)#5{\begingroup
	\mycoupon(#1,#2)(#3,#4)
	\afterassignment\ignoreuntilrelax\diagcouponXa=#1\psxunit\relax
	\afterassignment\ignoreuntilrelax\diagcouponXb=#3\psxunit\relax
	\advance\diagcouponXa\diagcouponXb\divide\diagcouponXa\tw@\relax
	\afterassignment\ignoreuntilrelax\diagcouponYa=#2\psyunit\relax
	\afterassignment\ignoreuntilrelax\diagcouponYb=#4\psyunit\relax
	\advance\diagcouponYa\diagcouponYb\divide\diagcouponYa\tw@\relax
	\edef\next{\noexpand\rput(\the\diagcouponXa,\the\diagcouponYa)}%
	\next{$\scriptstyle#5$}
\endgroup}
\makeatother

\def\ignoreuntilrelax#1\relax{}%

\def\pswall(#1,#2)(#3,#4){\begingroup
	\afterassignment\ignoreuntilrelax\dimen20= #1\psxunit\relax\advance\dimen20 0.1\psxunit
	\afterassignment\ignoreuntilrelax\dimen21=-#3\psxunit\relax\advance\dimen21-0.1\psxunit
	\rput(\the\dimen20,#2){%
		\psframe[fillstyle=vlines,linestyle=none,hatchsep=3pt](-0.1,0)(\the\dimen21,#4)%
		\psline(-0.1,0)(-0.1,#4)%
}\endgroup}

\psset{linewidth=1pt,xunit=0.7,dotsize=4pt}

\allowdisplaybreaks		

\newcommand{\setdoubleline}{\def\baselinestretch{1.2}\selectfont}
\newcommand{\setsingleline}{\def\baselinestretch{1.0}\selectfont}

\begin{document}

\title{The~Degenerate Distributive Complex is Degenerate}

\author{J\'ozef H.\ Przytycki}
\address{
	Department of Mathematics\\
	The George Washington University\\
	Washington, DC 20052\\
	\and
	University of Gda\'nsk, Poland
}
\email{przytyck@gwu.edu}
\thanks{%
	JHP was partially supported by the NSA-AMS 091111 grant, the GWU-REF grant, and Simons Collaboration Grant-316446.%
}

\author{	Krzysztof K.\ Putyra}
\address{
	Institute for Theoretical Studies\\
	ETH\\
	8092 Z\"urich, Switzerland
}
\email{krzysztof.putyra@eth-its.ethz.ch}
\thanks{%
	KKP was partially supported by the~Columbia University topology RTG grant DMS-0739392 and by the~NCCR SwissMAP%
}


\begin{abstract}
	We prove that the~degenerate part of the~distributive homology of a~multispindle is determined by the~normalized homology. In particular, when the~multispindle is a~quandle $Q$, the~degenerate homology of $Q$ is completely determined by the~quandle homology of $Q$. For this case (and generally for two-term homology of a~spindle) we provide an~explicit K\"unneth-type formula for the~degenerate part. This solves the~mystery in algebraic knot theory of the~meaning of the~degenerate quandle homology, brought over 15 years ago when the~homology theories were defined, and the~degenerate part was observed to be non trivial.
\end{abstract}

\maketitle

\setcounter{tocdepth}{1}

\setdoubleline

\section{Introduction}
Quandle homology \cites{Carter-Survey,Qndl-homology,Qndl-cocycle-invs,FennRourkeSand} is build in analogy to group homology or Hochschild homology of associate structures. In the~unital associative case we deal with simplicial sets (or modules) and it is a~classical result that the~degenerate part of a~chain complex is acyclic, so homology and normalized homology are isomorphic. It is not the~case for distributive structures, e.g.\ for quandles: we deal here only with a~week simplicial module \cite{Prz-distr-survey} and the~degenerate part can be not acyclic, as observed for quandles. In the~concrete case of quandle homology it is proven that the~homology (called the~\emph{rack homology}) splits into degenerate and normalized (called the~\emph{quandle homology}) parts \cite{LithNelson-splitting}, but no clear general connection between degenerate and quandle part were observed.

Quandles are special case of \emph{multispindles}, sets with a~number of self-distributive operations, for which analogous homology theory exists \cite{Prz-distr-survey}, and with a~right definition of a~module over a~multispindle one case also define homology with coefficients, generalizing twisted rack homology \cite{Twisted-qndl-hom}. In this paper we construct a~filtration on the~degenerate chain complex leading to a~spectral sequence with a~nice second page.

\begin{repeattheorem}{Theorem}{thm:spectral-sequence}
	Let a~multispindle $X$ acts on an~$R$-module $M$. Then there is a~spectral sequence $(E^r,\diff^r)$ converging to the~degenerate multiterm homology $\HD(M;X)$ such that $E^2_{pq} = \HN_p\big(\widehat H_{q-2}(M;X); X\big)$.
\end{repeattheorem}

Therefore, the~degenerate homology in degree $q$ is controlled by homology in degree less than $q$. This motivates to look for some recursive formula computing degenerate homology from the~normalized one, and such a~formula exists at least for (one-term) homology of spindles and (two-term) homology of quandles.

In the~first case we construct an~explicit isomorphism between the~degenerate chain complex and its associated graded complex with respect to the~filtration mentioned above, resulting in an~isomorphism
	\begin{equation*}
		\HD_n(M;X,\diff^{\opn}) \cong \bigoplus_{\mathclap{p+q=n}}
																\widehat H_{q-2}(M;X,\diff^{\opn})\otimes \CN_p(X)
	\end{equation*}
for any spindle $X$ (see Theorem~\ref{thm:one-term-HD}). In particular, the~spectral sequence mentioned above collapses at the~first page. This leads to a~recursive formula for degenerate homology, which can be easily solved. For instance, we obtain the~following decomposition.

\begin{repeattheorem}{Corollary}{cor:one-term-recursive-formula}
	If the~spindle $(X,\opn)$ is finite,
	\begin{equation*}
		\HD_n(X,\diff^{\opn}) \cong \bigoplus_{p=1}^n \Caugm\HN_{n-p}(X,\diff^{\opn})%
					^{\oplus \frac{|X|}{1+|X|}\left(|X|^p - (-1)^p\right)}.
	\end{equation*}
\end{repeattheorem}

In case of rack homology the~graded associated chain complex is not isomorphic to the~degenerate complex---the~first page on the~spectral sequence has a~nonvanishing differential. Instead the~degenerate chain complex $\CD(M;X)$ is isomorphism to the~tensor product $\widehat C(M;X)[2]\otimes\CN(X)$, leading to a~K\"unneth-type formula, see Theorem~\ref{thm:two-term-HD} for the~exact statement. The~immediate corollary is that the~normalized homology controls the~degenerate part.

\begin{repeattheorem}{Corollary}{cor:two-term-HD-from-HN}
	Suppose a~spindle homomorphism $\varphi\colon X\to X'$ induces an~isomorphism on normalized rack homology $\varphi_*\colon \HN(X,\diff^R) \longrightarrow \HN(X',\diff^R)$. If $M$ is an~$X'$-module such that the~induced map $\varphi_*\colon \HN(M^\varphi;X,\diff^R) \longrightarrow \HN(M;X',\diff^R)$ is also an~isomorphism, so is $\varphi_*\colon \HD(M^\varphi;X,\diff^R) \longrightarrow \HD(M;X',\diff^R)$.
\end{repeattheorem}

It is not immediate clear that the~same should hold for homology of all multispindles, even for those with two nontrivial operations. The~best we could achieve is that if a~homomorphism of multispindles induces an~isomorphism on normalized homology with coefficients in any module from a~certain class, then it induces an~isomorphism on degenerate homology.

\begin{repeattheorem}{Theorem}{thm:multiterm-HD-from-HN}
	Choose multispindles $X$, $X'$, an~$X'$-module $M$, and a~multispindle homomorphism $\varphi \colon X \longrightarrow X'$ inducing an~isomorphism $\varphi_*\colon \HN(M^\varphi;X) \longrightarrow \HN(M;X')$. If it also induces an~isomorphism $\varphi_* \colon \HN(N^\varphi;X) \longrightarrow \HN(N;X')$ for any~$X'$-module $N$ with a~vanishing \totalaction, then $\varphi_* \colon \HD(M^\varphi;X) \longrightarrow \HD(M;X')$ is an~isomorphism.
\end{repeattheorem}

We prove it by showing that such a~homomorphism induces an~isomorphism on the~second page of the~spectral sequence. The~inductive argument is based on the~following result on spectral sequences, which we have not encountered before.

\begin{repeattheorem}{Lemma}{lem:part-isom-on-E2-gives-isom-on-Eoo}
	Assume there is a~homomorphism of spectral sequences $f^r_{pq} \colon E^r_{pq} \longrightarrow \bar E^r_{pq}$ such that $f^2_{pq}$ is an~isomorphism for $q \leqslant N$. Then $f^{\infty}_{pq}\colon E^{\infty}_{pq} \longrightarrow \bar E^{\infty}_{pq}$ is an~isomorphism for $p+q=N$.
\end{repeattheorem}

Because the~proof is very technical and not related directly to distributive homology, we decided to move it to the~end of the~paper.

\subsection*{Outline}
The~paper is organized as follows. We provide basic definitions and results on distributive and rack homology in Section~\ref{sec:definitions}, and homology with coefficients in modules over mutlispindles is introduced in Section~\ref{sec:coeffs}. Most of the~computation is done using a~graphical calculus explained briefly in Section~\ref{sec:calculus}. The~next three sections contain the~main results of the~paper: construction of the~spectral sequence for degenerate homology in Section~\ref{sec:filtration}, and the~analysis of the~special cases of one-term (Section~\ref{sec:one-term}) and rack homology (Section~\ref{sec:two-term}). We provide a~very brief discussion on spectral sequences in Appendix~\ref{sec:spectral}, indluding the~proof of the~technical Lemma~\ref{lem:part-isom-on-E2-gives-isom-on-Eoo}.

\section{Distributive and rack homology}\label{sec:definitions}
\begin{definition}
	A~\emph{spindle} is a~set $X$ with a~binary operation $\opn\colon X\times X\to X$ that is idempotent and self-distributive from the~right side, i.e.\ $x\opn x=x$ and $(x\opn y)\opn z = (x\opn z)\opn(y\opn z)$ for any $x,y,z\in X$. If the~function $x\mapsto x\opn y$ is invertible for any $y\in X$ we say $(X,\opn)$ is a~\emph{quandle}. By dropping the~idempotent condition we obtain respectively a~\emph{shelf} and a~\emph{rack}.
\end{definition}

The~names shelf and spindle were coined by A.~Crans in her PhD thesis \cite{Crans-2-algebras} and are broadly used by knot theorists. The~older names, used outside knot theory, are respectively a~\emph{right distributive system} and a~\emph{right distributive idempotent system}, see \cite{Braids-and-self-distr}.

\begin{definition}
	Choose a~set $X$ with a~number of spindle operations $\opn_1,\dots,\opn_r$. We say $(X;\opn_1,\dots,\opn_r)$ is a~\emph{multispindle} if the~operations are mutually distributive, i.e.\ $(x\opn_i y)\opn_j z = (x\opn_j z)\opn_i(y\opn_j z)$ for any $x,y,z\in X$. We define \emph{multishelves} likewise.
\end{definition}

\begin{remark}
	Because the~trivial operation $x\rtriv y = x$ is distributive with respect to any shelf operation \cites{Prz-distr-survey,PrzPut-lattices}, one can extend any shelf $(X,\opn)$ to a~mutlishelf $(X;\opn,\rtriv)$.
\end{remark}

Given a~spindle $(X,\opn)$ and a~ring $R$ we define $C_n(X):=R\langle X^{n+1}\rangle$ to be the~$R$-module generated freely by all sequences $(x_n,\dots,x_0)\in X^{n+1}$. We shall write $\seq x$ for such a~sequence and we define $|\seq x|:=n$. The~\emph{(one-term) distributive differential} $\diff^{\opn}\colon C_n(X)\longrightarrow C_{n-1}(X)$ is given as the~alternating sum of \emph{face maps} $d^{\opn}_i\colon C_n(X)\longrightarrow C_{n-1}(X)$:
\begin{align}\label{eq:distr-diff}
	&\diff^{\opn} = \displaystyle{\sum_{i=0}^n (-1)^{n-i}d^{\opn}_i},\\
	\notag
	&d^{\opn}_i(x_n,\dots,x_0) := (x_n\opn x_i,\dots,x_{i+1}\opn x_i,\ x_{i-1},\dots,x_0).
\end{align}
The~unusual sign convention for $\diff^{\opn}$ is the~result of enumerating elements in a~sequence $\seq x$ from right to left, contrary to the~standard practice \cites{Qndl-cocycle-invs,Qndl-homology,Prz-distr-survey,PrzSik-distr-hom}. We check the~presimplicial relation $d^{\opn}_id^{\opn}_j = d^{\opn}_{j-1}d^{\opn}_i$ for $i>j$ (see also Example~\ref{ex:graph-faces}), from which it follows $(\diff^{\opn})^2 = 0$. We call the~homology of this chain complex the~\emph{(one-term) distributive homology} of $X$.

Given a~multishelf $(X;\opn_1,\dots,\opn_r)$ one can check that $\diff^{\opn_i}\diff^{\opn_j} + \diff^{\opn_j}\diff^{\opn_i} = 0$. This guarantees that any linear combination $\diff = \sum_{k=1}^r a_k\diff^{\opn_k}$ is a~differential on $C(X)$, which we call the~\emph{multiterm distributive differential} with weights $(a_1,\dots,a_r)$. Of particular interest is the~case of the~\emph{rack differential} $\diff^R = \diff^{\rtriv}-\diff^{\opn}$, where $x\rtriv y := x$ \cites{Qndl-homology,FennRourkeSand}. Notice that our rack homology is shifted by one comparing to the~definition due to Fenn, Rourke, and Sanderson. The main reason is that we wanted to deal with a~pre-simplicial category while they chose the convention of a~pre-cubic category.

\begin{remark}
	We shall write $C(X)$ for the~chain complex and $H(X)$ for its homology if we do not want to specialize the~differential. Otherwise, the~notation $C(X,\diff^{\opn})$, $H(X,\diff^R)$, etc.\ will be used.
\end{remark}

\begin{definition}
	Let $X$ be a~multispindle. The~\emph{degenerate chain complex} $\CD(X)$ is the~subcomplex of $C(X)$ generated by sequences $\seq x$ with repetitions, i.e.\ $x_i=x_{i+1}$ for some $i$. The~quotient $\CN(X):=C(X)/\CD(X)$ is called the~\emph{normalized complex} of $X$. Homology of the~complexes are called respectively \emph{degenerate} and \emph{normalized}, and they are denoted by $\HD(X)$ and $\HN(X)$.
\end{definition}

\begin{remark}
	In case of quandles, the~normalized homology is usually referred to as the~\emph{quandle homology} and is written as $H^Q(X)$, see \cite{Qndl-homology}.
\end{remark}

The~following theorem was first proven for the~rack homology of quandles \cite{LithNelson-splitting}, and then extended to the~one- and multiterm case \cite{Prz-distr-survey}.

\begin{theorem}[\cites{LithNelson-splitting,Prz-distr-survey}]\label{thm:CN-split}
	Given a~multispindle $X$, the~following short exact sequence
	\begin{equation}\label{SES:degen}
		0 \longrightarrow\CD(X)
			\longrightarrow C(X)
			\longrightarrow\CN(X)
			\longrightarrow 0.
	\end{equation}
	splits in a~canonical way. In particular, $H(X)\cong\HN(X)\oplus\HD(X)$.
\end{theorem}

In case of rack homology sequences $\seq x$ with late repetitions (i.e.\ $x_i=x_{i-1}$ for some $i<n$ in $(x_n,\dots,x_0)$) form a~chain subcomplex $\CL(X,\diff^R)\subset\CD(X,\diff^R)$, which is a~direct summand. It is called the~\emph{late degenerate complex} in \cite{LithNelson-splitting}.

\begin{theorem}[\cite{LithNelson-splitting}]\label{thm:CL-split}
	There is a~short exact sequence that splits canonically
	\begin{equation}\label{SES:late-degen}
		0 \longrightarrow\CL(X,\diff^R)
			\longrightarrow\CD(X,\diff^R)
			\longrightarrow\CN(X,\diff^R)[1]
			\longrightarrow 0.
	\end{equation}
	In particular, $\HN_n(X,\diff^R)$ is a~direct summand of $\HD_{n+1}(X,\diff^R)$.
\end{theorem}

\noindent
The~number in brackets indicates a~homological shift: $C[k]_n := C_{n-k}$.

In the~general case of multiterm homology $\CL(X)$ is a~subcomplex of $\CD(X)$ if and only if the~sum of all weights is zero, in which case the~proof from \cite{LithNelson-splitting} of the~theorem above holds. We shall generalize it to homology with coefficients in the~next section.

\section{Homology with coefficients}\label{sec:coeffs}
The~following definition is motivated by the~notion of an~$X$-set, introduced by S.~Kamada.

\begin{definition}\label{def:action-of-spindle}
	An~\emph{action} of a~spindle $(X,\opn)$ on an~$R$-module $M$ is a~function $\opn\colon M\times X\longrightarrow M$ that is linear in the~first variable and
	\begin{equation}\label{eq:action-condition}
		(m\opn x)\opn y = (m\opn y)\opn (x\opn y)
	\end{equation}
	for any $m\in M$ and $x,y\in X$. If $(X,\opn)$ is a~quandle we require also that it acts by automorphisms of $M$, i.e.\ the~function $m\mapsto m\opn x$ is invertible for every $x\in X$.
\end{definition}

An~$R$-module $M$ carrying an~action of a~spindle $X$ will be often called an~$X$-module.

\begin{example}
	Consider a~spindle $(X,\rtriv)$ with the~trivial operation $x\rtriv y = x$. Then the~actions on $M$ of all elements of $X$ commute:
	\begin{equation}
		(m\rtriv x)\rtriv y = (m\rtriv y)\rtriv(x\rtriv y) = (m\rtriv y)\rtriv x.
	\end{equation}
	Hence, $X$-modules are precisely modules over the~polynomial algebra $\mathbb Z[X]$ with as many variables as there are elements in $X$.
\end{example}

We generalize Definition~\ref{def:action-of-spindle} to multispindles in a~natural way.

\begin{definition}
	An~action of a~multispindle $(X;\opn_1,\dots,\opn_r)$ on an~$R$-module $M$ consists of functions $\opn_i\colon M\times X\to M$ linear in the~first variable, such that
	\begin{equation}\label{eq:multi-action-condition}
		(m\opn_i x)\opn_j y = (m\opn_j y)\opn_i(x\opn_j y)
	\end{equation}
	for any $m\in M$, $x,y\in X$ and $i,j=1,\dots,r$.
\end{definition}

Given a~multispindle $(X;\opn_1,\dots,\opn_r)$, which acts on $M$, choose $a_1,\dots,a_r\in R$. We define the~\emph{\totalaction} of $X$ on $M$ with weights $(a_1,\dots,a_r)$ as the~linear combination
\begin{equation}
	(m,x) \longmapsto m\cdot x := \sum_{i=1}^r a_i(m\opn_i x).
\end{equation}
A~direct computation shows the~\totalaction\ is distributive with respect to the~action of $X$ on $M$, i.e.\ $(m\cdot x)\opn_i y = (m\opn_i y)\cdot(x\opn_i y)$ for $i=1,\dots,r$.

\begin{example}
	Every multispindle $X$ acts on any module $M$ trivially, $m\opn_i x := m$, in which case the~\totalaction\ is the~multiplication by the~sum of weights.
\end{example}

\begin{example}
	A~multispindle $X$ acts on the~distributive chain complex $C(X)$ by acting on each element of a~sequence: $(x_n,\dots,x_0)\opn_i y := (x_n\opn_i y,\dots,x_0\opn_i y)$. This action descends to homology and it was observed in \cite{PrzPut-lattices} that $\alpha\cdot w=0$ for any homology class $\alpha\in H(X)$ and $w\in X$. Indeed, $\alpha\cdot w = \diff h^w + h^w\diff$ for $h^w(\seq x) := (-1)^{|\seq x|}(\seq x,w)$.
\end{example}

Let a~multispindle $X$ acts on an~$R$-module $M$. We define the~\emph{homology with coefficients in $M$} by setting $C_n(M;X) := M\otimes C_n(X)$ and $d^{\opn_k}_i(m\otimes\seq x) := (m\opn_k x_i)\otimes d^{\opn_k}_i\seq x$. Clearly, $C(X) = C(R;X)$ with the~trivial action of $X$ on $R$. As before, we have the~degenerate subcomplex with $\CD_n(M;X):=M\otimes\CD_n(X)$ and the~normalized one $\CN(M;X) := C(M;X)/\CD(M;X)$.

\begin{proposition}\label{prop:CN-M-split}
	The~short exact sequence splits canonically
	\begin{equation}\label{SES:M-degen}
		0 \longrightarrow\CD(M;X)
			\longrightarrow C(M;X)
			\longrightarrow\CN(M;X)
			\longrightarrow 0.
	\end{equation}
	In particular, $H(M;X)\cong\HN(M;X)\oplus\HD(M;X)$.
\end{proposition}
\begin{proof}
	Let $\alpha\colon\CN(X)\to C(X)$ be the~splitting map from Theorem~\ref{thm:CN-split}. Because each $C_n(X)$ is a~free $R$-module, $\CN_n(M;X) = M\otimes\CN_n(X)$ and an~easy computation shows the~map $\alpha^M := \id\otimes\alpha$ splits the~sequence~\eqref{SES:M-degen}.
\end{proof}

The~case of modules $M$ with a~vanishing \totalaction\ is very special. As before let $\CL_n(M;X):=M\otimes\CL_n(X)$ be spanned by late degenerate sequences.

\begin{proposition}
	If the~\totalaction\ of $X$ annihilates $M$, then $\CL(M;X)$ is a~subcomplex of $\CD(M;X)$ and there is an~isomorphism $s\colon \CN(M;X)[1] \longrightarrow \CD(M;X) / \CL(M;X)$ given by the~formula $s(m\otimes(x_n,\dots x_0)) = (-1)^nm \otimes (x_n,x_n,\dots,x_0)$.
\end{proposition}
\begin{proof}
	Vanishing of the~\totalaction\ implies vanishing of the~\totalface\ map $d_n\colon C_n(M;X)\longrightarrow C_{n-1}(M;X)$, $d_n = \sum_k a_kd^{\opn_k}_n$. Hence, $\CL(M;X)$ is a~subcomplex of $\CD(M;X)$. For $s$ to be an~isomorphism we have to check it is a~chain map, which follows from the~equality $d^{\opn_k}_n=d^{\opn_k}_{n-1}$ on $\CD_n(M;X)/\CL(M;X)$. Indeed, every sequence in the~quotient $\CD(M;X)/\CL(M;X)$ begins with a~repetition.
\end{proof}

\begin{remark}
	One can follow the~proof of Theorem~\ref{thm:CL-split} from \cite{LithNelson-splitting} to see that $\CL(M;X)$ is a~direct summand of $\CD(M;X)$.
\end{remark}

\begin{remark}
	A~multispindle $X$ acts naturally on its homology with coefficients. The~\totalaction\ vanishes, which is proven using the~homotopy $h^w(m\otimes\seq x) := (-1)^{|\seq x|}m \otimes (\seq x,w)$.
\end{remark}

The~one- and multiterm chain complexes can be augmented with $C_0(X) \stackrel \epsilon \longrightarrow R$ sending each generator to $\epsilon(x):=1$. This new chain complex and its homology are written as $\widetilde C(X)$ and $\widetilde H(X)$ respectively. We redefine it for complexes with coefficients in an~$X$-module by appending to the~chain complex $C(M;X)$ the~linearized \totalaction\ $C_0(M;X)=M \otimes R \langle X \rangle \longrightarrow M$. We write $\widehat C(M;X)$ and $\widehat H(M;X)$ for this chain complex and its homology respectively.

\begin{remark}
	$\widehat H_n(M;X)=H_n(M;X)$ for $n\geqslant 0$ and $\widehat H_{-1}(M;X) = M$ when the~\totalaction\ vanishes. In particular, $\widehat H(R;X)$ is different from $\widetilde H(R;X)$ in case of rack homology.
\end{remark}

We construct the~augmented normalized complex $\Chat\CN(M;X)$ with homology $\Chat\HN(M;X)$ likewise. It is worth to notice that Proposition~\ref{prop:CN-M-split} still holds for augmented complexes.

\section{Graphical calculus}\label{sec:calculus}
Consider a~picture in a~plane consisting of a~number of strands originating on a~horizontal line and going \upordown{upwards}{downwards} to another horizontal line. The~strands can cross with each other\footnote{
	These are real crossings on a~plane but we can interprete them with no harm as \upordown{negative}{positive} crossings.
}
and a~single strand can terminate in a~dot but there are no turnbacks. Given a~spindle $(X,\opn)$ we can interprete such a~picture as a~map $X^n\longrightarrow X^m$, where $n$ and $m$ are numbers of enpoints of strands at the~\upordown{lower and upper}{upper and lower} horizontal lines respectively. Namely, decorate the~$n$ endpoints on the~\upordown{lower}{upper} line with elements of $X$ and propagate the~labels \upordown{upwards}{downwards} along the~strands. A~label is forgotten at a~terminal dot, and each time we encounter a~crossing the~right label propagates with no change but the~left one is replaced by the~product of both, see Fig.~\ref{fig:diagram-as-a-map}.

\begin{figure}
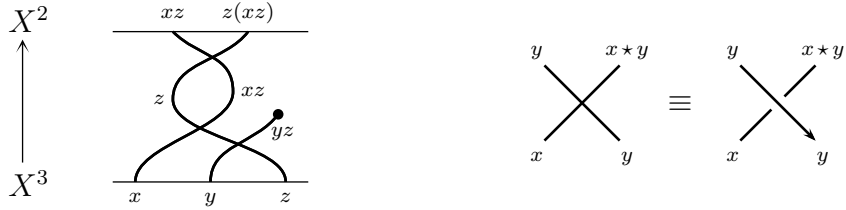
%
	\psset{xunit=1cm,yunit=1cm,labelsep=3pt}
	\begin{centerpict}(-1.5,-0.4)(2.3,2.5)
		\rput[B](-1.5,-1ex){\Rnode{dom}{$X\mathrlap{{}^3}$}}
		\rput[B](-1.5,2){\Rnode{cod}{$X\mathrlap{{}^2}$}}
		\ncline[nodesep=3pt,linewidth=0.5pt,arrowsize=5pt,arrowlength=1,arrowinset=0.7]{->}{dom}{cod}
		\psline[linewidth=0.5\pslinewidth](-0.3,0)(2.3,0)
		\psline[linewidth=0.5\pslinewidth](-0.3,2)(2.3,2)
		\pscustom{%
				\moveto(0,0)
				\curveto(0,0.5)(1.3,0.7)(1.3,1.2)
				\curveto(1.3,1.7)(0.85,1.65)(0.5,2)
				\stroke}%
				\uput[r](1.3,1.2){$\scriptstyle xz$}
		\pscustom{%
				\moveto(1,0)
				\curveto(1,0.5)(1.6,0.5)(1.9,0.9)
				\stroke}%
				\psdot(1.9,0.9)
				\uput[r](1.7,0.65){$\scriptstyle yz$}
		\pscustom{%
				\moveto(2,0)
				\curveto(2,0.5)(0.5,0.6)(0.5,1.1)
				\curveto(0.5,1.6)(1.2,1.6)(1.5,2.0)
				\stroke}%
				\uput[l](0.5,1.1){$\scriptstyle z$}
		\rput(0,-1.5ex){%
			\rput[B](0,0){$\scriptstyle x$}
			\rput[B](1,0){$\scriptstyle y$}
			\rput[B](2,0){$\scriptstyle z$}
		}
		\rput(0,1ex){%
			\rput[B](0.5,2){$\scriptstyle xz$}
			\rput[B](1.5,2){$\scriptstyle z(xz)$}
		}
	\end{centerpict}
	\hskip 3cm
	\begin{centerpict}(0,0)(1.5,1)
		\psline(0,0)(1,1)
		\psline(1,0)(0,1)
		\rput(0,-1.5ex){%
			\rput[B](-0.1,0){$\scriptstyle x$}
			\rput[B]( 1.1,0){$\scriptstyle y$}
		}\rput(0,0.8ex){%
			\rput[B](-0.1,1){$\scriptstyle y$}
			\rput[B]( 1.1,1){$\scriptstyle x\opn y$}
		}
	\end{centerpict}
	$\equiv$
	\begin{centerpict}(-0.5,0)(1,1)
		\psline{-}(0,0)(1,1)
		\psline[border=3\pslinewidth]{<-}(1,0)(0,1)
		\rput(0,-1.5ex){%
			\rput[B](-0.1,0){$\scriptstyle x$}
			\rput[B]( 1.1,0){$\scriptstyle y$}
		}\rput(0,0.8ex){%
			\rput[B](-0.1,1){$\scriptstyle y$}
			\rput[B]( 1.1,1){$\scriptstyle x\opn y$}
		}
	\end{centerpict}
	\caption{%
		A~string diagram seen as a~function $X^3\longrightarrow X^2$. The~pictures to the~right explain how labels propagates through a~crossing and where the~motivation comes from.
	}\label{fig:diagram-as-a-map}%
\end{figure}

Different pictures can encode the~same map. For instance, far away crossings and terminal dots commute:
\begin{equation}\label{graph:far-away-commute}
	\psset{xunit=5mm,yunit=8mm}
	\begin{centerpict}(0,0)(4,1.5)
		\pscustom{%
			\moveto(0,0)\lineto(0,0.5)\curveto(0,1.0)(1,1.0)(1,1.5)
			\moveto(1,0)\lineto(1,0.5)\curveto(1,1.0)(0,1.0)(0,1.5)
		}
		\rput(2,0.75){$\cdots$}
		\rput(3,0){\pscustom{%
			\moveto(0,0)\curveto(0,0.5)(1,0.5)(1,1.0)\lineto(1,1.5)
			\moveto(1,0)\curveto(1,0.5)(0,0.5)(0,1.0)\lineto(0,1.5)
		}}
	\end{centerpict}
	\quad=\quad
	\begin{centerpict}(0,0)(4,1.5)
		\pscustom{%
			\moveto(0,0)\curveto(0,0.5)(1,0.5)(1,1.0)\lineto(1,1.5)
			\moveto(1,0)\curveto(1,0.5)(0,0.5)(0,1.0)\lineto(0,1.5)
		}
		\rput(2,0.75){$\cdots$}
		\rput(3,0){\pscustom{%
			\moveto(0,0)\lineto(0,0.5)\curveto(0,1.0)(1,1.0)(1,1.5)
			\moveto(1,0)\lineto(1,0.5)\curveto(1,1.0)(0,1.0)(0,1.5)
		}}
	\end{centerpict}
	\hskip 2cm
	\begin{centerpict}(0,0)(3,1.5)
		\pscustom{%
			\moveto(0,0)\lineto(0,0.5)\curveto(0,1.0)(1,1.0)(1,1.5)
			\moveto(1,0)\lineto(1,0.5)\curveto(1,1.0)(0,1.0)(0,1.5)
		}
		\rput(2,0.75){$\cdots$}
		\psline(3,0)(3,0.5)\psdot(3,0.5)
	\end{centerpict}
	\quad=\quad
	\begin{centerpict}(0,0)(3,1.5)
		\pscustom{%
			\moveto(0,0)\curveto(0,0.5)(1,0.5)(1,1.0)\lineto(1,1.5)
			\moveto(1,0)\curveto(1,0.5)(0,0.5)(0,1.0)\lineto(0,1.5)
		}
		\rput(2,0.75){$\cdots$}
		\psline(3,0)(3,1)\psdot(3,1)
	\end{centerpict}
	\hskip 1cm\textrm{etc.}
\end{equation}
and a~terminal dot can be pulled through a~crossing from the~left hand side but not from the~right one:
\begin{equation}\label{graph:dot-vs-crossing}
	\psset{xunit=1cm,yunit=1cm}
	\begin{centerpict}(0,0)(1,1)
		\psline(1,0)(0,1)
		\psline(0,0)(0.75,0.75)\psdot(0.75,0.75)
	\end{centerpict}
	\quad=\quad
	\begin{centerpict}(0,0)(1,1)
		\psline(1,0)(0,1)
		\psline(0,0)(0.25,0.25)\psdot(0.25,0.25)
	\end{centerpict}\ \raisebox{-1.5ex}{,}
	\hskip 1cm\text{but}\hskip 1cm
	\begin{centerpict}(0,0)(1,1)
		\psline(0,0)(1,1)
		\psline(1,0)(0.25,0.75)\psdot(0.25,0.75)
	\end{centerpict}
	\quad\neq\quad
	\begin{centerpict}(0,0)(1,1)
		\psline(0,0)(1,1)
		\psline(1,0)(0.75,0.25)\psdot(0.75,0.25)
		\rput(1.2,0.1){.}
	\end{centerpict}
\end{equation}
The~most interesting are the~following two relations
\begin{equation}\label{diag:spindle-axioms}
	\psset{yunit=\psxunit}
	\begin{centerpict}[midline=0.7](0,-1.5ex)(1,1.4)
		\psbezier(0,0)(0,0.6)(1,0.8)(1,1.4)
		\psbezier(1,0)(1,0.6)(0,0.8)(0,1.4)
		\rput[B](0,-1.5ex){$\scriptstyle x$}
		\rput[B](1,-1.5ex){$\scriptstyle x$}
	\end{centerpict}
	\quad=\quad
	\begin{centerpict}[midline=0.7](0,-1.5ex)(1,1.4)
		\psline(0,0)(0,1.4)
		\psline(1,0)(1,1.4)
		\rput[B](0,-1.5ex){$\scriptstyle x$}
		\rput[B](1,-1.5ex){$\scriptstyle x$}
	\end{centerpict}
	\qquad\textnormal{and}\qquad
	\begin{centerpict}(0,0)(2,2)
		\psline(0,0)(2,2)
		\psline(2,0)(0,2)
		\psbezier(1,0)(-0.3,0.8)(-0.3,1.2)(1,2)
	\end{centerpict}
	\quad=\quad
	\begin{centerpict}(0,0)(2,2)
		\psline(0,0)(2,2)
		\psline(2,0)(0,2)
		\psbezier(1,0)(2.3,0.8)(2.3,1.2)(1,2)
	\end{centerpict}\quad,
\end{equation}
which are equivalent to the~axioms of a~spindle. The~first one holds only when both inputs are labeled with the~same element, and it visualizes the~idempotency axiom. The~right picture is the~famous Reidemeister III move, also called the~\emph{braid} or the~\emph{Yang-Baxter relation}. It holds for any input and is equivalent to the~self-distributivity of $\opn\colon X\times X\longrightarrow X$.

If $X$ acts on a~module $M$, add a~wall to the~left of the~picture---it can be labeled with elements $m\in M$. A~strand can terminate on the~wall, which corresponds to the~action map $M\times X\longrightarrow M$. The~condition for an~action of $X$ translates as absorbing a~crossing by the~wall:
\begin{equation}
	\begin{centerpict}(-0.3,0)(1.6,1.4)
		\pswall(0,0)(0.3,1.4)
		\psbezier(0.75,0)(0.75,1)(0.5,1)(0,1.2)
		\psbezier(1.5,0)(1.5,0.5)(1.0,0.5)(0,0.8)
	\end{centerpict}
	\quad=\quad
	\begin{centerpict}(-0.3,0)(1.7,1.4)
		\pswall(0,0)(0.3,1.4)
		\psbezier(0.75,0)(0.75,0.5)(0.5,0.5)(0,0.8)
		\psbezier(1.5,0)(1.5,0.6)(1,0.7)(0,1.2)
	\end{centerpict}
\end{equation}

\begin{example}\label{ex:graph-faces}
	\wrapfigure[r]{%
		$d^{\opn}_i := \begin{centerpict}[midline=0.5](-0.5,-1.7ex)(2.5,1)
			\pswall(0,0)(0.3,1)
			\psbezier(0.5,0)(0.5,0.5)(1.0,0.5)(1.0,1)
			\psbezier(1.0,0)(1.0,0.5)(1.5,0.5)(1.5,1)
			\psbezier(1.5,0)(1.5,0.5)(0.5,0.5)(0.0,0.7)	
			\psline(2.0,0)(2.0,1)
			\psline(2.5,0)(2.5,1)
			\rput[B](0.5,-1.7ex){$\scriptstyle n$}
			\rput[B](1.5,-1.7ex){$\scriptstyle i$}
			\rput[B](2.5,-1.7ex){$\scriptstyle 0$}
		\end{centerpict}$}
	The~$i$-th face map $d^{\opn}_i\colon C_n(M;X)\longrightarrow C_{n-1}(M;X)$ can be visualized graphically by pulling the~$i$-th strand all the~way to the~left, and terminating it at the~wall. Then the~presimplicial relation follows from an~easy deformation of pictures. For instance,
	\begin{equation}
		d^{\opn}_id^{\opn}_j\,=\,
		\begin{centerpict}[midline=0.5](-1,-2ex)(2.7,1)
			\pswall(-0.5,0)(0.3,1)
			\psline(0.0,0)(0.0,1)
			\psline(0.5,0)(0.5,1)
			\psline(2.0,0)(2.0,1)
			\psline(2.5,0)(2.5,1)
			\pscustom{\moveto(1.0,0)\lineto(1.0,0.1)\curveto(1.0,0.35)(-0.3,0.4)(-0.5,0.5)}
			\pscustom{\moveto(1.5,0)\lineto(1.5,0.3)\curveto(1.5,0.55)(-0.3,0.7)(-0.5,0.8)}
			\rput[B](0.0,-1.7ex){$\scriptstyle n$}
			\rput[B](1.0,-1.7ex){$\scriptstyle j$}
			\rput[B](1.5,-1.7ex){$\scriptstyle i$}
			\rput[B](2.5,-1.7ex){$\scriptstyle 0$}
		\end{centerpict}
		\,=\,
		\begin{centerpict}[midline=0.5](-1.2,-2ex)(2.7,1)
			\pswall(-0.7,0)(0.3,1)
			\psline(0.0,0)(0.0,1)
			\psline(0.5,0)(0.5,1)
			\psline(2.0,0)(2.0,1)
			\psline(2.5,0)(2.5,1)
			\pscustom{\moveto(1.0,0)\lineto(1.0,0.1)
				\curveto( 1.0,0.3)(0,0.3)(-0.2,0.35)
				\curveto(-0.4,0.4)(-0.3,0.8)(-0.7,0.9)}
			\pscustom{\moveto(1.5,0)\lineto(1.5,0.3)\curveto(1.5,0.55)(-0.4,0.6)(-0.7,0.6)}
			\rput[B](0.0,-1.7ex){$\scriptstyle n$}
			\rput[B](1.0,-1.7ex){$\scriptstyle j$}
			\rput[B](1.5,-1.7ex){$\scriptstyle i$}
			\rput[B](2.5,-1.7ex){$\scriptstyle 0$}
		\end{centerpict}
		\,=\,
		\begin{centerpict}[midline=0.5](-1,-2ex)(2.7,1)
			\pswall(-0.5,0)(0.3,1)
			\psline(0.0,0)(0.0,1)
			\psline(0.5,0)(0.5,1)
			\psline(2.0,0)(2.0,1)
			\psline(2.5,0)(2.5,1)
			\pscustom{\moveto(1.0,0)\lineto(1.0,0.4)\curveto(1.0,0.65)(-0.3,0.7)(-0.5,0.8)}
			\pscustom{\moveto(1.5,0)\lineto(1.5,0.1)\curveto(1.5,0.35)(-0.3,0.4)(-0.5,0.5)}
			\rput[B](0.0,-1.7ex){$\scriptstyle n$}
			\rput[B](1.0,-1.7ex){$\scriptstyle j$}
			\rput[B](1.5,-1.7ex){$\scriptstyle i$}
			\rput[B](2.5,-1.7ex){$\scriptstyle 0$}
		\end{centerpict}
		\,= d^{\opn}_{j-1} d^{\opn}_i,
	\end{equation}
	when $j>i$.
\end{example}

\begin{example}\label{ex:graph-faces-triv}
	\wrapfigure[r]{%
		$d^{\rtriv}_i := \begin{centerpict}[midline=0.5](-0.5,-2ex)(2.5,1)
			\pswall(-0.1,0)(0.2,1)
			\psline(0.5,0)(0.5,1)
			\psline(1.0,0)(1.0,1)
			\psline(1.5,0)(1.5,0.5)\psdot(1.5,0.5)
			\psline(2.0,0)(2.0,1)
			\psline(2.5,0)(2.5,1)
			\rput(0,-1.7ex){%
				\rput[B](0.5,0){$\scriptstyle n$}
				\rput[B](1.5,0){$\scriptstyle i$}
				\rput[B](2.5,0){$\scriptstyle 0$}
			}
		\end{centerpict}$}
	In the~case of the~trivial operation $x\rtriv y = x$, we can use a~simpler picture for the~face map $d^{\rtriv}_i$: the~$i$-th strand is immediately terminated with a~dot. The~presimplicial relation follows trivially, and in the~mixed case we use \eqref{graph:dot-vs-crossing}. For instance, when $j>i$, one computes
	\begin{equation}
		d^{\opn}_id^{\rtriv}_j\,=\,
		\begin{centerpict}[midline=0.5](-1,-2ex)(2.7,1)
			\pswall(-0.5,0)(0.3,1)
			\psline(0.0,0)(0.0,1)
			\psline(0.5,0)(0.5,1)
			\psline(1.0,0)(1.0,0.3)\psdot(1.0,0.3)
			\psline(2.0,0)(2.0,1)
			\psline(2.5,0)(2.5,1)
			\pscustom{\moveto(1.5,0)\lineto(1.5,0.3)\curveto(1.5,0.55)(-0.3,0.7)(-0.5,0.8)}
			\rput[B](0.0,-1.7ex){$\scriptstyle n$}
			\rput[B](1.0,-1.7ex){$\scriptstyle j$}
			\rput[B](1.5,-1.7ex){$\scriptstyle i$}
			\rput[B](2.5,-1.7ex){$\scriptstyle 0$}
		\end{centerpict}
		\,=\,
		\begin{centerpict}[midline=0.5](-1,-2ex)(2.7,1)
			\pswall(-0.5,0)(0.3,1)
			\psline(0.0,0)(0.0,1)
			\psline(0.5,0)(0.5,1)
			\psbezier(1.0,0)(1.0,0.3)(1.25,0.4)(1.25,0.7)\psdot(1.25,0.7)
			\psline(2.0,0)(2.0,1)
			\psline(2.5,0)(2.5,1)
			\pscustom{\moveto(1.5,0)\lineto(1.5,0.15)\curveto(1.5,0.4)(-0.3,0.5)(-0.5,0.6)}
			\rput[B](0.0,-1.7ex){$\scriptstyle n$}
			\rput[B](1.0,-1.7ex){$\scriptstyle j$}
			\rput[B](1.5,-1.7ex){$\scriptstyle i$}
			\rput[B](2.5,-1.7ex){$\scriptstyle 0$}
		\end{centerpict}
		= d^{\rtriv}_{j-1} d^{\opn}_i.
	\end{equation}
\end{example}

We shall often restrict to sequences with a~repetition at a~certain place, which will be visualized by a~band joining two strands: it forces its two edges to carry the~same label. For instance, we can visualize generators of $\CD(M;X)$ by vertical lines with a~band at some position:
\begin{equation}
	\CD_n(M;X) := \Big\langle\begin{centerpict}[midline=0.3](-0.5,-1.7ex)(3.2,0.6)
		\bandframe(1.5,0.05)(2.0,0.55)
		\pswall(0.0,0)(0.3,0.6)
		\psline(0.5,0)(0.5,0.6)
		\psline(1.0,0)(1.0,0.6)
		\psline(1.5,0)(1.5,0.6)
		\psline(2.0,0)(2.0,0.6)
		\psline(2.5,0)(2.5,0.6)
		\psline(3.0,0)(3.0,0.6)
		\rput(0,-1.7ex){%
			\rput[B](0.5,0){$\scriptstyle n$}
			\rput[B](2.0,0){$\scriptstyle i$}
			\rput[B](3.0,0){$\scriptstyle 0$}
		}
	\end{centerpict}: 0\leqslant i < n\,\Big\rangle,
\end{equation}
and the~idempotency axiom can be rewritten without specifying labels at the~input as
\begin{equation}
	\psset{unit=\psxunit}
	\begin{centerpict}(0,0)(1,1.4)
		\psclip{\psframe[linestyle=none](0,0.1)(1,0.7)}%
			\bandcustom{%
				\moveto(0,0)
				\curveto(0,0.6)(1,0.8)(1,1.4)\lineto(0,1.4)
				\curveto(0,0.8)(1,0.6)(1,0.0)\lineto(0,0.0)
			}%
		\endpsclip
		\psbezier(0,0)(0,0.6)(1,0.8)(1,1.4)
		\psbezier(1,0)(1,0.6)(0,0.8)(0,1.4)
	\end{centerpict}
	\quad=\quad
	\begin{centerpict}(0,0)(1,1.4)
		\bandframe(0,0.1)(1,1.3)
		\psline(0,0)(0,1.4)
		\psline(1,0)(1,1.4)
		\rput(1.2,0.1){.}
	\end{centerpict}
\end{equation}
Strands that are not joined by a~band can still carry the~same label, however. Finally, the~map $s_p\colon C_p(M;X)\longrightarrow C_{p+1}(M;X)$ can be visualized by thickening the~left most strand into a~band
\begin{equation}
	s_p = (-1)^p
	\begin{centerpict}[midline=0.4](-0.5,-2ex)(3.2,0.8)
		\pspolygon[fillstyle=solid,fillcolor=bandcolor,linestyle=none]
			(0.75,0.2)(0.5,0.4)(0.5,0.75)(1.0,0.75)(1.0,0.4)
		\pswall(0.0,0)(0.3,0.8)
		\psline(0.75,0)(0.75,0.2)(0.5,0.4)(0.5,0.8)
		\psline(0.75,0)(0.75,0.2)(1.0,0.4)(1.0,0.8)
		\psline(1.5,0)(1.5,0.8)
		\psline(2.0,0)(2.0,0.8)
		\psline(2.5,0)(2.5,0.8)
		\psline(3.0,0)(3.0,0.8)
		\rput(0,-1.7ex){%
			\rput[B](0.75,0){$\scriptstyle p$}
			\rput[B](2.0,0){$\scriptstyle i$}
			\rput[B](3.0,0){$\scriptstyle 0$}
		}
		\rput(3.4,0.1){.}
	\end{centerpict}
\end{equation}

\begin{remark}
	We shall use the~same graphical calculus for multispindles without mixing different operations, i.e.\ we shall always use one operation to interpret a~diagram.
\end{remark}

\section{A~filtration of the~degenerate complex}\label{sec:filtration}
Let $\F^p_{\!n}\subset\CD_n(M;X)$ be the~submodule generated by elements $m\otimes\seq x$, where the~sequence $\seq x$ has a~repetition at position $p$ or closer to the~right, i.e.\ $x_i=x_{i+1}$ for some $i\leqslant p$. Diagrammatically,
\begin{equation}
	\F^p_{\!n} := \Big\langle\begin{centerpict}[midline=0.3](-0.5,-1.7ex)(3.5,0.6)
		\bandframe(1.5,0.05)(2.0,0.55)
		\pswall(0.0,0)(0.3,0.6)
		\psline(0.5,0)(0.5,0.6)
		\psline(1.0,0)(1.0,0.6)
		\psline(1.5,0)(1.5,0.6)
		\psline(2.0,0)(2.0,0.6)
		\psline(2.5,0)(2.5,0.6)
		\psline(3.0,0)(3.0,0.6)
		\rput(0,-1.7ex){%
			\rput[B](0.5,0){$\scriptstyle n$}
			\rput[B](2.0,0){$\scriptstyle i$}
			\rput[B](3.0,0){$\scriptstyle 0$}
		}
	\end{centerpict}: 0\leqslant i \leqslant p\,\Big\rangle.
\end{equation}

\begin{proposition}\label{prop:filtration}
	Each $\F^p$ is a~chain subcomplex of $\CD(M;X)$, and all together they form a~filtration of $\CD(M;X)$, i.e.\ $\F^p\subset \F^{p+1}$ for every $p$, and $\displaystyle{\bigcup_{p\geqslant 0}\F^p = \CD(M;X)}$.
\end{proposition}
\begin{proof}
	The~two conditions for a~filtration follow directly from the~definition of $\F^p_n$. To show that $\F^p$ is closed under the~differential, choose a~sequence $\seq x$ with a~repetition at position $i$. Then for any operation $\opn_k$ on $X$ the~faces $d^{\opn_k}_i(m\otimes\seq x) = d^{\opn_k}_{i+1}(m\otimes\seq x)$ cancel each other, whereas the~other faces have repetitions at positions $i$ or $i-1$:
	\begin{equation}
		\begin{centerpict}[midline=0.5](-0.8,-2ex)(3.5,\dimexpr\psyunit+2ex)
			\bandline(1.75,0.1)(1.75,0.9)
			\pswall(-0.5,0)(0.3,1)
			\psline(0.0,0)(0.0,1)
			\psline(0.5,0)(0.5,1)
			\psbezier(1,0)(1,0.8)(0,0.6)(-0.5,0.8)
			\psline(1.5,0)(1.5,1)
			\psline(2.0,0)(2.0,1)
			\psline(2.5,0)(2.5,1)
			\psline(3.0,0)(3.0,1)
			\psline(3.5,0)(3.5,1)
			\rput[B](3.5,-1.7ex){$\scriptstyle 0$}
			\rput[B](2.0,-1.7ex){$\scriptstyle i$}
			\rput[B](1.0,-1.7ex){$\scriptstyle j$}
			\rput(0,1){%
				\rput[B](3.5,0.8ex){$\scriptstyle 0$}
				\rput[B](2.0,0.8ex){$\scriptstyle i$}
			}
		\end{centerpict}
		\hskip 1cm\textnormal{or}\hskip 1cm
		\begin{centerpict}[midline=0.5](-0.8,-2ex)(3.5,\dimexpr\psyunit+2ex)
			\bandline(1.75,0.1)(1.75,0.9)
			\pswall(-0.5,0)(0.3,1)
			\psline(0.0,0)(0.0,1)
			\psline(0.5,0)(0.5,1)
			\psline(1.0,0)(1.0,1)
			\psline(1.5,0)(1.5,1)
			\psline(2.0,0)(2.0,1)
			\psline(2.5,0)(2.5,1)
			\psbezier(3,0)(3,0.8)(0.4,0.6)(-0.5,0.8)
			\psline(3.5,0)(3.5,1)
			\rput(0,-1.7ex){%
				\rput[B](3.5,0){$\scriptstyle 0$}
				\rput[B](2.0,0){$\scriptstyle i$}
				\rput[B](3.0,0){$\scriptstyle j$}
			}
			\rput(0,1){%
				\rput[B](3.5,0.8ex){$\scriptstyle 0$}
				\rput[B](2.0,0.8ex){$\scriptstyle i{-}1$}
			}
		\end{centerpict}\quad.
	\end{equation}
\vskip -2.5ex\end{proof}

Because the~filtration is given on generators, the~quotients $\F^p/\F^{p-1}$ are easy to understand: each is generated by elements $m\otimes\seq x$, where the~sequence $\seq x$ has its first repetition occurring at position $p$, when looking from the~right hand side. As $d^{\opn_k}_i(m\otimes\seq x) \in \F^{p-1}$ if $i<p$, only higher faces survive.

\begin{corollary}\label{cor:one-term-quotient}
	There is an~isomorphism of chain complexes
	\begin{equation}\label{eq:filter-quotient}
		\F^p/\F^{p-1} \cong \widehat C(M;X)[p+2]\otimes \CN_p(X),
	\end{equation}
	where on the~right hand side the~differential acts only on the~first factor.
\end{corollary}
\begin{proof}
	The~face maps $d^{\opn_k}_i$ vanish on the~quotient for $i<p$, and since $d^{\opn_k}_p$ cancels $d^{\opn_k}_{p+1}$,
	\begin{equation}
		\begin{centerpict}[midline=0.5](-0.5,-2ex)(3.75,1)
			\bandline(1.75,0.05)(1.75,0.5)
			\pswall(0.0,0)(0.5,1)
			\psline(0.5,0)(0.5,0.5)\psline(0.75,0.5)(0.75,1)
			\psline(1.0,0)(1.0,0.5)\psline(1.25,0.5)(1.25,1)
			\psline(1.5,0)(1.5,0.5)\psline(1.75,0.5)(1.75,1)
			\psline(2.0,0)(2.0,0.5)\psline(2.25,0.5)(2.25,1)
			\psline(2.5,0)(2.5,0.5)\psline(2.75,0.5)(2.75,1)
			\psline(3.0,0)(3.0,0.5)\psline(3.25,0.5)(3.25,1)
			\psline(3.5,0)(3.5,0.5)
			\rput[B](3.5,-1.7ex){$\scriptstyle 0$}
			\rput[B](2.0,-1.7ex){$\scriptstyle p$}
			\rput[B](0.5,-1.7ex){$\scriptstyle n$}
			\diagcoupon(-0.25,0.3)(3.75,0.8){\diff}
		\end{centerpict}
		\quad=\quad
		\begin{centerpict}[midline=0.5](-0.5,-2ex)(3.75,1)
			\bandline(1.75,0.05)(1.75,0.95)
			\pswall(0.0,0)(0.5,1)
			\psline(0.5,0)(0.5,0.5)\psline(0.75,0.5)(0.75,1)
			\psline(1.0,0)(1.0,0.5)
			\psline(1.5,0)(1.5,1)
			\psline(2.0,0)(2.0,1)
			\psline(2.5,0)(2.5,1)
			\psline(3.0,0)(3.0,1)
			\psline(3.5,0)(3.5,1)
			\rput[B](3.5,-1.7ex){$\scriptstyle 0$}
			\rput[B](2.0,-1.7ex){$\scriptstyle p$}
			\rput[B](0.0,-1.7ex){$\scriptstyle n$}
			\diagcoupon(-0.25,0.25)(1.25,0.75){\diff}
		\end{centerpict}.
	\end{equation}
	Hence, the~repetition splits $\seq x$ into two parts: the~left one, which can consists only of the~wall (thence the~augmented complex in \eqref{eq:filter-quotient}), and the~right one (including the~$p$-th strand), which has no repetition. The~latter is preserved by the~differential.
\end{proof}

Recall that the~\emph{graded associated} chain complex $\gr\CD(M;X)$ with respect to the~filtration $\F^p$ is the~direct sum of quotients $\F^p/\F^{p-1}$.

\begin{corollary}
	Let $(X;\opn_1,\dots,\opn_r)$ be a~multispindle acting on a~module $M$. Then
	\begin{equation}
		H_n(\gr\CD(M;X)) \cong \bigoplus_{p+q=n} \widehat H_{q-2}(M;X)\otimes\CN_p(X).
	\end{equation}
\end{corollary}

The~filtration $\F^p$ leads to a~spectral sequence (see Appendix~\ref{sec:spectral}) computing the~degenerate homology of $X$, and the~corollary above shows its first page. Its second page is the~normalized homology of $X$ with coefficients in the~augmented homology.

\begin{theorem}\label{thm:spectral-sequence}
	Let a~multispindle $X$ acts on an~$R$-module $M$. Then there is a~spectral sequence $(E^r,\diff^r)$ converging to the~degenerate multiterm homology $\HD(M;X)$ such that $E^2_{pq} = \HN_p\big(\widehat H_{q-2}(M;X); X\big)$.
\end{theorem}
\begin{proof}
	Due to Corollary~\ref{cor:one-term-quotient} there is an~isomorphism
	\begin{displaymath}
		\id\otimes s_p\colon \widehat H_{q-2}(M;X)\otimes\CN_p(X)
			\stackrel\cong\longrightarrow E^1_{pq},
	\end{displaymath}	
	so that the~differential $\diff^1_{pq}\colon E^1_{pq}\longrightarrow E^1_{p-1,q}$ computes homology of $s\CN(X) \cong \CD(X)/\CL(X)$ with coefficients in $\widehat H_{q-2}(M;X)$. Since the~\totalaction\ of $X$ on $\widehat H_{q-2}(M;X)$ vanishes,
	\begin{displaymath}
		s\colon\CN(\widehat H_{q-2}(M;X); X)\longrightarrow s\CN(\widehat H_{q-2}(M;X); X)
	\end{displaymath}
	is actually a~chain map, thence an~isomorphism.
\end{proof}

\begin{corollary}
	Assume $\Chat\HN(M;X)=0$. Then degenerate homology vanishes. In particular, if the~augmented homology of a~spindle is trivial, so is its degenerate homology.
\end{corollary}
\begin{proof}
	Suppose $\HD_i(M;X)=0$ for $i<N$. Then $E^r_{pq}=\HN_p(\widehat H_{q-2}(M;X);X)=0$ for $q<N+2$. In particular, $E^{\infty}_{pq}=0$ when $p+q=N$, which implies $\HD_N(M;X)=0$.
\end{proof}

\begin{corollary}
	Assume the~augmented quandle homology (i.e.\ $\diff=\diff^{\ltriv}-\diff^{\opn}$) of a~spindle $X$ is trivial. Then $\HD_n(R;X)=R$ for $n>0$.
\end{corollary}
\begin{proof}
	The~compound action on $R$ vanishes, so that $\Chat\HN_0(R;X)=\Chat\HN_{-1}(R;X)=R$ and $\Chat\HN_p(R;X)=\HN_p(R;X)=0$ for $p>0$. Hence, $E^2_{1,0}=0$ and $E^2_{0,1}=R$, which implies $\HD_1(R;X)=E^2_{0,1}=R$. Suppose by induction that $\HD_i(R;X)=R$ for $i=1,\dots,N-1$. Then $E^2_{pq} = \HN_p(R;X)$ for $q<N+2$, which is zero except $E^2_{0,q}=H_0(R;X)=R$. Again, it must be $\HD_N(R;X)=E^2_{0,N}=R$.
\end{proof}

Choose a~homomorphism of spindles $\varphi\colon X\to X'$ and an~$X'$-module $M$. Then $\varphi$ induces an~action of $X$ on $M$, $m\star_i x := m\star_i \varphi(x)$, and we shall write $M^\varphi$ for this $X$-module.

\begin{corollary}\label{cor:HD-from-HN(H)}
	Suppose a~homomorphism of multispindles $\varphi\colon X\longrightarrow X'$ induces isomorphisms on homology groups $\varphi_*\colon \HN(\widehat H_{q}(M^\varphi;X); X) \longrightarrow \HN(\widehat H_{q}(M; X'); X')$ for any $q$. Then $\varphi_*\colon \HD(M^\varphi;X)\longrightarrow \HD(M;X')$ is an~isomorphism.
\end{corollary}

Assume that $\varphi\colon X\longrightarrow X'$ given as above induces also an~isomorphism on normalized homology $\varphi_* \colon \HN(M^\varphi;X) \longrightarrow \HN(M;X')$. Then $H(M^\varphi;X) \cong H(M;X')^\varphi$ as $X$-modules and we can restate the~assumptions on $\varphi$ requiring that it induces an isomorphism on normalized homology with coefficient in any module with a~vanishing compound action. Since only $\widehat H_k(M;X)$ with $k\leqslant n-2$ appear in $E^2_{pq}$ with $p+q=n$, one should be able to recover the~original assumption of Corollary~\ref{cor:HD-from-HN(H)} by an~induction argument. This is the~idea underlying the~following theorem.

\begin{theorem}\label{thm:multiterm-HD-from-HN}
	Choose multispindles $X$, $X'$, an~$X'$-module $M$, and a~multispindle homomorphism $\varphi \colon X \longrightarrow X'$ inducing an~isomorphism $\varphi_*\colon \HN(M^\varphi;X) \longrightarrow \HN(M;X')$. If it also induces an~isomorphism $\varphi_* \colon \HN(N^\varphi;X) \longrightarrow \HN(N;X')$ for any~$X'$-module $N$ with a~vanishing \totalaction, then $\varphi_* \colon \HD(M^\varphi;X) \longrightarrow \HD(M;X')$ is an~isomorphism.
\end{theorem}

\begin{proof}
	Consider the~following diagram with exact rows:
	\begin{displaymath}
		\xymatrix{%
 			  0 \ar[r] & M^\varphi[-1]           \ar[r]\ar[d]^{\id}
				         & \widehat C(M^\varphi;X) \ar[r]\ar[d]^{\widehat\varphi_*}
								 & C(M^\varphi;X)          \ar[r]\ar[d]^{\varphi_*}
								 & 0 \\
			  0 \ar[r] & M[-1]            \ar[r]
				         & \widehat C(M;X') \ar[r]
								 & C(M;X')          \ar[r]
								 & 0
		}
	\end{displaymath}
	Due to the~5-lemma, $\widehat\varphi_*$ induces an~isomorphism on homology if and only if so does $\varphi_*$.

	Denote by $E^r$ and $\bar E^r$ the~spectral sequences for $\HD(M^\varphi;X)$ and $\HD(M;X')$. Since $\widehat H_k(M^\varphi;X)=\Chat\HN_k(M^\varphi;X)$ and $\widehat H_k(M;X')=\Chat\HN_k(M;X')$ for $k<1$, we immediately see that $f^2_{pq} \colon E^2_{pq} \longrightarrow \bar E^2_{pq}$ is an~isomorphism for $q<3$ (homology is a~module with a~vanishing compound action). It follows now from Lemma~\ref{lem:part-isom-on-E2-gives-isom-on-Eoo} that $f^{\infty}_{pq}$ is an~isomorphism if $p+q=1$, so that $\HD_1(M^\varphi;X) \cong \HD_1(M;X')$. Hence, $f^2_{pq}$ is an~isomorphism for $q<4$ and similarly we get $\HD_2(M^\varphi;X) \cong \HD_2(M;X')$. Use induction to finish the~proof.
\end{proof}

In the~next two sections we shall strengthen this result for the~one-term and rack homology, obtaining recursive formulas for degenerate homology in terms of the~normalized one.

\section{Degenerate spindle homology}\label{sec:one-term}
The~filtration $\F^p$ of $\CD(M;X)$ is given on generators, so that the~chain groups $\gr\CD_n(M;X)$ and $\CD_n(M;X)$ are naturally isomorphic. We shall identify them together. However, this identification is not compatible with differentials, and our goal is to find correcting terms to obtain a~chain map $f\colon\gr\CD_n(M;X)\longrightarrow\CD_n(M;X)$ that is \emph{filtered}, i.e.\ we require $f$ to send $\gr_p\CD(M;X)$ into $\F^p$. The~following lemma is a~classical result from the~theory of filtered modules \cite{Spectral-sequences}.

\begin{lemma}
	Let $C$ and $D$ be chain complexes with filtrations $\F^pC$ and $\F^pD$ respectively, and let $f\colon C\longrightarrow D$ be a~filtered chain map, i.e.\ $f(\F^pC)\subset \F^pD$ for any $p$. If\/ $\gr f\colon\gr C\longrightarrow\gr D$ is an~isomorphism, so is $f$.
\end{lemma}

\begin{corollary}\label{cor:filtered-isom}
	A~filtered chain map $f\colon\gr\CD(M;X)\longrightarrow\CD(M;X)$ is an~isomorphism if and only if $\gr_p\CD(M;X)\stackrel{f}\longrightarrow \F^p\stackrel{pr}\longrightarrow \F^p/\F^{p-1}$ is an~isomorphism for every $p$.
\end{corollary}

\noindent
We define $\ungrade\colon\gr\CD(M;X)\longrightarrow\CD(M;X)$ by correcting the~identity homomorphisms with lower order terms. Namely, the~component $\ungrade^p\colon \gr_p\CD(M;X) \longrightarrow \F^p$ is given by the~following picture
\begin{equation}\label{eq:def-fp}
	\ungrade^p :=
	\begin{centerpict}[midline=0.6](-0.3,-2ex)(3,1.2)
			\bandline(1.75,0.05)(1.75,0.6)
			\pswall(0.0,0)(0.3,1.2)
			\psline(0.5,0)(0.5,1.2)
			\psline(1.0,0)(1.0,1.2)
			\psline(1.5,0)(1.5,1.2)
			\psline(2.0,0)(2.0,1.2)
			\psline(2.5,0)(2.5,1.2)
			\psline(3.0,0)(3.0,1.2)
			\diagcoupon(1.25,0.35)(3.25,0.9){u^p}
			\rput[B](3.0,-1.7ex){$\scriptstyle 0$}
			\rput[B](2.0,-1.7ex){$\scriptstyle p$}
			\rput[B](0.5,-1.7ex){$\scriptstyle n$}
			\rput(3.5,0.1){,}
	\end{centerpict}
\end{equation}
where the~map $u^p\colon C_{p+1}(X)\longrightarrow C_{p+1}(X)$ is defined using the~recursive formula
\begin{equation}\label{eq:def-of-u}
		u^0 :=
		\begin{centerpict}(-0.2,0)(0.7,1)
			\psline(0.0,0)(0.0,1)
			\psline(0.5,0)(0.5,1)
		\end{centerpict}
		\hskip 2cm
		u^p :=
		\begin{centerpict}(-0.2,0)(1.7,1.5)
			\psline(0.0,0)(0.0,1.5)
			\psline(0.5,0)(0.5,1.5)
			\psline(1.0,0)(1.0,1.5)
			\psline(1.5,0)(1.5,1.5)
			\diagcoupon(-0.25,0.4)(1.25,1.1){u^{p-1}}
		\end{centerpict}
		+ (-1)^p
		\begin{centerpict}(-0.2,0)(1.7,1.5)
			\psbezier(0.0,0)(0.3,0.4)(0.5,0.4)(0.5,0.8)\psline(0.5,0.8)(0.5,1.5)
			\psbezier(0.5,0)(0.8,0.4)(1.0,0.4)(1.0,0.8)\psline(1.0,0.8)(1.0,1.5)
			\psbezier(1.0,0)(1.3,0.4)(1.5,0.4)(1.5,0.8)\psline(1.5,0.8)(1.5,1.5)
			\psbezier(1.5,0)(1.5,0.2)(0.0,0.2)(0.0,0.6)\psline(0.0,0.6)(0.0,1.5)
			\diagcoupon(0.25,0.6)(1.75,1.3){u^{p-1}}
			\rput(1.9,0.1){.}
		\end{centerpict}
\end{equation}

\begin{theorem}\label{thm:one-term-HD}
	The~map $\ungrade\colon\gr\CD(M;X,\diff^{\opn})\longrightarrow\CD(M;X,\diff^{\opn})$ is a~natural isomorphism of filtered complexes. In~particular, there is a~natural isomorphism
	\begin{equation}\label{eq:HD-one-term}
		\HD_n(M;X,\diff^{\opn}) \cong \bigoplus_{\mathclap{p+q=n}}
																\widehat H_{q-2}(M;X,\diff^{\opn})\otimes \CN_p(X)
	\end{equation}
	for any spindle $(X,\opn)$.
\end{theorem}
\begin{proof}
	Naturality of $\ungrade$ is clear from the~way it is constructed, and we shall show it is a~chain map. First, observe that we can pull $u^p$ through a~line:
	\begin{equation}\label{eq:u-vs-line}
		\begin{centerpict}(-0.5,0)(2.4,1.5)
			\psline(0.0,0)(0.0,0.7)\psbezier(0.0,0.7)(0.0,1.1)(0.2,1.1)(0.5,1.5)
			\psline(0.5,0)(0.5,0.7)\psbezier(0.5,0.7)(0.5,1.1)(0.7,1.1)(1.0,1.5)
			\psline(1.0,0)(1.0,0.7)\psbezier(1.0,0.7)(1.0,1.1)(1.2,1.1)(1.5,1.5)
			\psline(1.5,0)(1.5,0.7)\psbezier(1.5,0.7)(1.5,1.1)(1.7,1.1)(2.0,1.5)
			\psline(2.0,0)(2.0,0.9)\psbezier(2.0,0.9)(2.0,1.3)(0.0,1.3)(0.0,1.5)
			\diagcoupon(-0.25,0.2)(1.75,0.9){u^p}
		\end{centerpict}
		=
		\begin{centerpict}(-0.4,0)(2.5,1.5)
			\psbezier(0.0,0)(0.3,0.4)(0.5,0.4)(0.5,0.8)\psline(0.5,0.8)(0.5,1.5)
			\psbezier(0.5,0)(0.8,0.4)(1.0,0.4)(1.0,0.8)\psline(1.0,0.8)(1.0,1.5)
			\psbezier(1.0,0)(1.3,0.4)(1.5,0.4)(1.5,0.8)\psline(1.5,0.8)(1.5,1.5)
			\psbezier(1.5,0)(1.8,0.4)(2.0,0.4)(2.0,0.8)\psline(2.0,0.8)(2.0,1.5)
			\psbezier(2.0,0)(2.0,0.2)(0.0,0.2)(0.0,0.6)\psline(0.0,0.6)(0.0,1.5)
			\diagcoupon(0.25,0.6)(2.25,1.3){u^p}
			\rput(2.7,0.1){.}
		\end{centerpict}
	\end{equation}
	This follows directly from \eqref{eq:def-of-u} by induction on the~number of strands. In~particular, we can pull the~right most line to the~left either over or below $u^{p-1}$ in the~formula \eqref{eq:def-of-u}. Then a~simple induction shows each component $\ungrade^p$ is a~chain map:
	\begin{align*}
		\diff^{\opn} \ungrade^p &=
		\begin{centerpict}[midline=1](-1,-2ex)(3,2)
				\pswall(-0.5,0)(0.5,2)
				\bandline(1.25,0)(1.25,0.5)
				\psline(0.0,0)(0.0,1.5)\psline(0.25,1.5)(0.25,2)
				\psline(0.5,0)(0.5,1.5)\psline(0.75,1.5)(0.75,2)
				\psline(1.0,0)(1.0,1.5)\psline(1.25,1.5)(1.25,2)
				\psline(1.5,0)(1.5,1.5)\psline(1.75,1.5)(1.75,2)
				\psline(2.0,0)(2.0,1.5)\psline(2.25,1.5)(2.25,2)
				\psline(2.5,0)(2.5,1.5)
				\diagcoupon( 0.75,0.3)(2.25,0.9){u^{p-1}}
				\diagcoupon(-0.75,1.1)(2.75,1.7){\diff^{\opn}}
				\rput[B](2.5,-1.7ex){$\scriptstyle 0$}
				\rput[B](1.5,-1.7ex){$\scriptstyle p$}
				\rput[B](0.0,-1.7ex){$\scriptstyle n$}
		\end{centerpict}
		+ (-1)^p
		\begin{centerpict}[midline=1](-1,-2ex)(3,2)
				\pswall(-0.5,0)(0.5,2)
				\psline(0.0,0)(0.0,1.5)\psline(0.25,1.5)(0.25,2)
				\psline(0.5,0)(0.5,1.5)\psline(0.75,1.5)(0.75,2)
				\psline(1.25,1.5)(1.25,2)
				\psline(1.75,1.5)(1.75,2)
				\psline(2.25,1.5)(2.25,2)
				\pscustom[linestyle=none,fillcolor=bandcolor,fillstyle=solid]{%
					\moveto(1.0,0.0)\curveto(1.3,0.4)(1.5,0.4)(1.5,0.8)
					\lineto(2.0,0.8)\curveto(2.0,0.4)(1.8,0.4)(1.5,0.0)
					\lineto(1.0,0.0)
				}
				\psbezier(1.0,0)(1.3,0.4)(1.5,0.4)(1.5,0.8)\psline(1.5,0.8)(1.5,1.5)
				\psbezier(1.5,0)(1.8,0.4)(2.0,0.4)(2.0,0.8)\psline(2.0,0.8)(2.0,1.5)
				\psbezier(2.0,0)(2.3,0.4)(2.5,0.4)(2.5,0.8)\psline(2.5,0.8)(2.5,1.5)
				\psbezier(2.5,0)(2.5,0.2)(1.0,0.1)(1.0,0.5)\psline(1.0,0.5)(1.0,1.5)
				\diagcoupon( 1.25,0.4)(2.75,1.0){u^{p-1}}
				\diagcoupon(-0.75,1.2)(2.75,1.8){\diff^{\opn}}
				\rput[B](2.5,-1.7ex){$\scriptstyle 0$}
				\rput[B](1.5,-1.7ex){$\scriptstyle p$}
				\rput[B](0.0,-1.7ex){$\scriptstyle n$}
		\end{centerpict}
		\\[2ex]
		&=
		\begin{centerpict}[midline=0.75](-1,-2ex)(3,1.5)
				\pswall(-0.5,0)(0.5,1.5)
				\bandline(1.25,0)(1.25,0.5)
				\psline(0.0,0)(0.0,1.0)\psline(0.25,1.0)(0.25,1.5)
				\psline(0.5,0)(0.5,1.0)
				\psline(1.0,0)(1.0,1.5)
				\psline(1.5,0)(1.5,1.5)
				\psline(2.0,0)(2.0,1.5)
				\psline(2.5,0)(2.5,1.5)
				\diagcoupon( 0.8,0.3)(2.2,0.9){u^{p-1}}
				\diagcoupon(-0.75,0.6)(0.7,1.2){\diff^{\opn}}
				\rput[B](2.5,-1.7ex){$\scriptstyle 0$}
				\rput[B](1.5,-1.7ex){$\scriptstyle p$}
				\rput[B](0.0,-1.7ex){$\scriptstyle n$}
		\end{centerpict}
		+(-1)^n
		\begin{centerpict}[midline=0.75](-1,-2ex)(3,1.5)
				\pswall(-0.5,0)(0.5,1.5)
				\bandline(1.25,0)(1.25,0.5)
				\psline(0.0,0)(0.0,0.7)\psbezier(0.0,0.7)(0.0,1.1)(0.25,1.1)(0.25,1.5)
				\psline(0.5,0)(0.5,0.7)\psbezier(0.5,0.7)(0.5,1.1)(0.75,1.1)(0.75,1.5)
				\psline(1.0,0)(1.0,0.7)\psbezier(1.0,0.7)(1.0,1.1)(1.25,1.1)(1.25,1.5)
				\psline(1.5,0)(1.5,0.7)\psbezier(1.5,0.7)(1.5,1.1)(1.75,1.1)(1.75,1.5)
				\psline(2.0,0)(2.0,0.7)\psbezier(2.0,0.7)(2.0,1.1)(2.25,1.1)(2.25,1.5)
				\psline(2.5,0)(2.5,0.7)\psbezier(2.5,0.7)(2.5,1.3)(0.5,1.1)(-0.5,1.3)
				\diagcoupon( 0.75,0.2)(2.25,0.8){u^{p-1}}
				\rput[B](2.5,-1.7ex){$\scriptstyle 0$}
				\rput[B](1.5,-1.7ex){$\scriptstyle p$}
				\rput[B](0.0,-1.7ex){$\scriptstyle n$}
		\end{centerpict}
		+(-1)^p
		\begin{centerpict}[midline=0.75](-1,-2ex)(3,1.5)
				\pswall(-0.5,0)(0.5,1.5)
				\psline(0.0,0)(0.0,1)\psline(0.25,1)(0.25,1.5)
				\psline(0.5,0)(0.5,1)\psline(0.75,1)(0.75,1.5)
				\pscustom[linestyle=none,fillcolor=bandcolor,fillstyle=solid]{%
					\moveto(1.0,0.0)\curveto(1.3,0.4)(1.5,0.4)(1.5,0.8)
					\lineto(2.0,0.8)\curveto(2.0,0.4)(1.8,0.4)(1.5,0.0)
					\lineto(1.0,0.0)
				}
				\psbezier(1.0,0)(1.3,0.4)(1.5,0.4)(1.5,0.8)\psline(1.5,0.8)(1.5,1.5)
				\psbezier(1.5,0)(1.8,0.4)(2.0,0.4)(2.0,0.8)\psline(2.0,0.8)(2.0,1.5)
				\psbezier(2.0,0)(2.3,0.4)(2.5,0.4)(2.5,0.8)\psline(2.5,0.8)(2.5,1.5)
				\psbezier(2.5,0)(2.5,0.2)(1.0,0.1)(1.0,0.5)\psline(1.0,0.5)(1.0,1.0)
				\diagcoupon( 1.3,0.4)(2.7,1.0){u^{p-1}}
				\diagcoupon(-0.75,0.6)(1.2,1.2){\diff^{\opn}}
				\rput[B](2.5,-1.7ex){$\scriptstyle 0$}
				\rput[B](1.5,-1.7ex){$\scriptstyle p$}
				\rput[B](0.0,-1.7ex){$\scriptstyle n$}
		\end{centerpict}
		\\[2ex]
		&=
		\begin{centerpict}[midline=0.75](-1,-2ex)(3,1.5)
				\pswall(-0.5,0)(0.5,1.5)
				\bandline(1.25,0)(1.25,0.8)
				\psline(0.0,0)(0.0,0.8)\psline(0.25,0.8)(0.25,1.5)
				\psline(0.5,0)(0.5,0.8)
				\psline(1.0,0)(1.0,1.5)
				\psline(1.5,0)(1.5,1.5)
				\psline(2.0,0)(2.0,1.5)
				\psline(2.5,0)(2.5,1.5)
				\diagcoupon( 0.8,0.6)(2.7,1.2){u^p}
				\diagcoupon(-0.75,0.3)(0.7,0.9){\diff^{\opn}}
				\rput[B](2.5,-1.7ex){$\scriptstyle 0$}
				\rput[B](1.5,-1.7ex){$\scriptstyle p$}
				\rput[B](0.0,-1.7ex){$\scriptstyle n$}
		\end{centerpict}
		+(-1)^n
		\begin{centerpict}[midline=0.75](-1,-2ex)(3,1.5)
				\pswall(-0.5,0)(0.5,1.5)
				\bandline(1.25,0)(1.25,0.5)
				\psline(0.0,0)(0.0,0.7)\psbezier(0.0,0.7)(0.0,1.1)(0.25,1.1)(0.25,1.5)
				\psline(0.5,0)(0.5,0.7)\psbezier(0.5,0.7)(0.5,1.1)(0.75,1.1)(0.75,1.5)
				\psline(1.0,0)(1.0,0.7)\psbezier(1.0,0.7)(1.0,1.1)(1.25,1.1)(1.25,1.5)
				\psline(1.5,0)(1.5,0.7)\psbezier(1.5,0.7)(1.5,1.1)(1.75,1.1)(1.75,1.5)
				\psline(2.0,0)(2.0,0.7)\psbezier(2.0,0.7)(2.0,1.1)(2.25,1.1)(2.25,1.5)
				\psline(2.5,0)(2.5,0.7)\psbezier(2.5,0.7)(2.5,1.3)(0.5,1.1)(-0.5,1.3)
				\diagcoupon( 0.75,0.2)(2.25,0.8){u^{p-1}}
				\rput[B](2.5,-1.7ex){$\scriptstyle 0$}
				\rput[B](1.5,-1.7ex){$\scriptstyle p$}
				\rput[B](0.0,-1.7ex){$\scriptstyle n$}
		\end{centerpict}
		+(-1)^{n-1}
		\begin{centerpict}[midline=0.75](-1,-2ex)(3,1.5)
				\pswall(-0.5,0)(0.5,1.5)
				\psline(0.0,0)(0.0,0.8)\psbezier(0.0,0.8)(0.0,1.1)(0.2,1.2)(0.25,1.5)
				\psline(0.5,0)(0.5,0.8)\psbezier(0.5,0.8)(0.5,1.1)(0.7,1.2)(0.75,1.5)
				\pscustom[linestyle=none,fillcolor=bandcolor,fillstyle=solid]{%
					\moveto(1.0,0.0)\curveto(1.3,0.4)(1.5,0.4)(1.5,0.8)
					\lineto(2.0,0.8)\curveto(2.0,0.4)(1.8,0.4)(1.5,0.0)
					\lineto(1.0,0.0)
				}
				\psbezier(1.0,0)(1.3,0.4)(1.5,0.4)(1.5,0.8)\psline(1.5,0.8)(1.5,1.5)
				\psbezier(1.5,0)(1.8,0.4)(2.0,0.4)(2.0,0.8)\psline(2.0,0.8)(2.0,1.5)
				\psbezier(2.0,0)(2.3,0.4)(2.5,0.4)(2.5,0.8)\psline(2.5,0.8)(2.5,1.5)
				\psbezier(2.5,0)(2.5,0.2)(1.0,0.1)(1.0,0.5)\psline(1.0,0.5)(1.0,0.8)
				\psbezier(1.0,0.8)(1.0,1.2)(0.0,1.2)(-0.5,1.3)
				\diagcoupon( 1.3,0.4)(2.7,1.0){u^{p-1}}
				\rput[B](2.5,-1.7ex){$\scriptstyle 0$}
				\rput[B](1.5,-1.7ex){$\scriptstyle p$}
				\rput[B](0.0,-1.7ex){$\scriptstyle n$}
		\end{centerpict}
		= \ungrade^p\diff^{\opn}.
	\end{align*}
	Finally, the~theorem follows from Corollary~\ref{cor:filtered-isom}, because $\ungrade^p(m\otimes\seq x)\in m\otimes\seq x+F^{p-1}$.
\end{proof}

According to the~theorem above every degenerate homology group splits into a~direct sum of copies of whole homology groups in lower degrees. These in turn split canonically into normalized and degenerate parts and the~latter can be recursively replaced with groups in lower degrees.

\begin{corollary}\label{cor:one-term-recursive-formula}
	If the~spindle $(X,\opn)$ is finite,
	\begin{equation}\label{eq:HD-in-terms-of-HN}
		\HD_n(X,\diff^{\opn}) \cong \bigoplus_{p=1}^n \Caugm\HN_{n-p}(X,\diff^{\opn})%
					^{\oplus \frac{|X|}{1+|X|}\left(|X|^p - (-1)^p\right)}.
	\end{equation}
\end{corollary}
\begin{proof}
	We proof the~formula by induction on $n$. It is clear for $n<2$, as both sides are trivial ($\HD_1(X,\diff^{\opn})=0$, since $(x,x)=\diff^{\opn}(x,x,x)$ for any $x\in X$). For bigger $n$ notice first that $\rk\CN_p(X) = |X|(|X|-1)^{p-1}$. In~particular, $\rk\CN_{p+1}(X) = (|X|-1)\rk\CN_p(X)$ and Theorem~\ref{thm:one-term-HD} implies
	\begin{equation}\begin{split}
		\HD_{n+1}(X,\diff^{\opn})
				&\cong \widetilde H_{n-1}(X,\diff^{\opn})^{\oplus|X|}
					\oplus \bigoplus_{p=2}^n\widetilde H_{n-p}(X,\diff^{\opn})^{\oplus|X|(|X|-1)^{p-1}}\\
				&\cong \Caugm\HN_{n-1}(X,\diff^{\opn})^{\oplus|X|}
					\oplus \HD_{n-1}(X,\diff^{\opn})^{\oplus|X|}
					\oplus \HD_n(X,\diff^{\opn})^{\oplus(|X|-1)}.
	\end{split}\end{equation}
	Using induction we show that $\Caugm\HN_{n-1}(X,\diff^{\opn})$ does not appear in the~second nor the~third term, and $\Caugm\HN_{n-2}(X,\diff^{\opn})$ comes only from the~last summand in multiplicity $|X|(|X|-1)$. For $p>2$, the~group $\Caugm\HN_{n+1-p}(X,\diff^{\opn})$ appears with multiplicity
	\begin{equation}
		\tfrac{|X|}{1+|X|}\Big(|X|(|X|^{p-2} - (-1)^p) + (|X|-1)(|X|^{p-1} + (-1)^p)\Big)
		=\tfrac{|X|}{1+|X|}\left(|X|^p - (-1)^p\right)
	\end{equation}
	as desired.
\end{proof}

The~above corollary shows that the~size of $\HD(X)$ is determined by $\HN(X)$ and the~size of the~spindle $X$. In fact, $\HD(X)$ determines $X$ unless $\Caugm\HN(X)=0$.

\begin{proposition}
	Suppose $\Caugm\HN_p(X,\diff^{\opn})\neq 0$ for some $p$. Then a~spindle homomorphism 
	$\varphi\colon X\longrightarrow X'$ induces an~isomorphism on degenerate homology if and only if it is an~isomorphism of spindles.
\end{proposition}
\begin{proof}
	Choose the~smallest $p$ for which $\Caugm\HN_p(X,\diff^{\opn})\neq 0$. Then
	\begin{equation}
		\HD_{p+1}(X,\diff^{\opn}) \cong
			\Caugm \HN_p(X,\diff^{\opn}) \otimes \CN_0(X,\diff^{\opn})
	\end{equation}	
	and the~map induced by $\varphi$ has the~form
	\begin{equation}
		\varphi_*\colon\HD_{p+1}(X,\diff^{\opn}) \longrightarrow \HN_{p+1}(X',\diff^{\opn}),
		\hskip 1cm
		(\ldots)\otimes x \mapsto (\ldots)\otimes \varphi(x).
	\end{equation}
	Hence, $\varphi$ must be bijective if $\varphi_*$ is an~isomorphism.
\end{proof}

\section{Degenerate generalized rack homology}\label{sec:two-term}
In the~case of the~rack differential $\diff^R := \diff^{\rtriv} - \diff^{\opn}$ the~degenerate complex is no longer isomorphic to $\gr\CD(M;X)$. Instead we shall show it is isomorphic to the~total complex of the~bicomplex $B_{pq}(M;X) := \widehat C_{q-2}(M;X)\otimes\CN_p(X)$. We filter the~bicomplex $B(M;X)$ by columns, i.e.\ $\F^p B:= \widehat C(M;X)\otimes\CN_p(X)$.

Due to Corollary~\ref{cor:one-term-quotient}, the~chain modules $B_{pq}(M;X)$ are isomorphic to quotients $\F^p_{p+q}/\F^{p-1}_{p+q}$, which can be seen as subgroups of $\CD_{p+q}(M;X)$. More precisely, there is a~family of monomorphisms $i_{pq}\colon B_{pq}(M;X)\longrightarrow\CD_{p+q}(M;X)$ defined as $i_{pq}(m\otimes\seq x\otimes\seq y) = m\otimes\seq x\otimes s(\seq y)$ and visualized by thickening the~$p$-th strand:
\begin{equation}
	i_{pq} := (-1)^p\begin{centerpict}[midline=0.4](-0.9,-2ex)(3.2,0.8)
		\pspolygon[fillstyle=solid,fillcolor=bandcolor,linestyle=none]%
			(1.5,0.75)(1.5,0.4)(1.75,0.2)(2.0,0.4)(2.0,0.75)
		\pswall(-0.5,0)(0.3,0.8)
		\psline(0.0,0)(0.0,0.8)
		\psline(0.5,0)(0.5,0.8)
		\psline(1.0,0)(1.0,0.8)
		\psline(1.75,0)(1.75,0.2)(1.5,0.4)(1.5,0.8)
		\psline(1.75,0)(1.75,0.2)(2.0,0.4)(2.0,0.8)
		\psline(2.5,0)(2.5,0.8)
		\psline(3.0,0)(3.0,0.8)
		\rput[B](1.75,-1.7ex){$\scriptstyle p$}
		\rput[B](3.0,-1.7ex){$\scriptstyle 0$}
		\rput(3.4,0.1){.}
	\end{centerpict}
\end{equation}
Recall that doubling the~left most element is coherent with the~rack differential, i.e.\ $\diff^R s(\seq x) = s(\diff^R\seq x)$. We now define a~family of maps $\bisom^p \colon \F^pB \longrightarrow \CD_{p+q}(M;X)$ by composing $i_{p\bullet}$ with the~map $u^p\colon C_{p+1}(X) \longrightarrow C_{p+1}(X)$ defined in the~previous secion and the~splitting homomorphism $\alpha \colon \CN(X) \longrightarrow C(X)$ acting on the~normalized factor:
\begin{equation}\label{eq:def-fp-two-term}
	\bisom^p := (-1)^p
	\begin{centerpict}[midline=1](-0.9,-2ex)(3.2,2)
			\pspolygon[fillstyle=solid,fillcolor=bandcolor,linestyle=none]%
				(1.5,1.4)(1.5,1.0)(1.75,0.85)(2.0,1.0)(2.0,1.4)
			\pswall(-0.5,0)(0.3,2.0)
			\psline(0.0,0)(0.0,2.0)
			\psline(0.5,0)(0.5,2.0)
			\psline(1.0,0)(1.0,2.0)
			\psline(1.75,0)(1.75,0.85)(1.5,1)(1.5,2.0)
			\psline(1.75,0)(1.75,0.85)(2.0,1)(2.0,2.0)
			\psline(2.5,0)(2.5,2.0)
			\psline(3.0,0)(3.0,2.0)
			\diagcoupon(1.25, 1.25)(3.25,1.8){u^p}
			\diagcoupon(1.5, 0.15)(3.25, 0.7){\alpha_p^{\phantom p}}
			\rput[B](3,-1.7ex){$\scriptstyle 0$}
			\rput[B](1.75,-1.7ex){$\scriptstyle p$}
	\end{centerpict}
\end{equation}
The~homomorphism $\alpha_p$ translates the~normalized rack differential into the~unnormalized one, which makes graphical calculus easier---we do not have to care about repetitions when applying $u^p$.

\begin{theorem}\label{thm:two-term-HD}
	The~map $\bisom\colon Tot(B(M;X))\longrightarrow\CD(M;X)$ is an~isomorphism of filtered complexes. In particular, if $R$ is a~p.i.d. there is a~short exact sequence
	\begin{multline}
		0	\longrightarrow \bigoplus_{\mathclap{p+q=n}}\widehat H_{q-2}(M;X,\diff^R)\otimes\HN_p(X,\diff^R)
			\longrightarrow \HD_n(M;X,\diff^R)\\
			\longrightarrow \bigoplus_{\mathclap{p+q=n-1}} Tor\left(\widehat H_{q-2}(M;X,\diff^R), \HN_p(X,\diff^R)\right)
			\longrightarrow 0
	\end{multline}
	which splits.
\end{theorem}
\begin{proof}
	As before, $\bisom^p(\seq x)\in\seq x+F^{p-1}$, which in the~view of Corollary~\ref{cor:filtered-isom} guarantees $\bisom$ is an~isomorphism if it is a~chain map. This follows immediately from Theorem~\ref{thm:one-term-HD}. Indeed,
	\begin{align*}
		\diff^{\rtriv}u^p &= 
		\begin{centerpict}[midline=1](0,-2ex)(3,2)
				\bandline(0.75,0)(0.75,0.5)
				\psline(0.5,0)(0.5,1.5)\psline(0.75,1.5)(0.75,2)
				\psline(1.0,0)(1.0,1.5)\psline(1.25,1.5)(1.25,2)
				\psline(1.5,0)(1.5,1.5)\psline(1.75,1.5)(1.75,2)
				\psline(2.0,0)(2.0,1.5)\psline(2.25,1.5)(2.25,2)
				\psline(2.5,0)(2.5,1.5)
				\diagcoupon(0.25,0.3)(2.25,0.9){u^{p-1}}
				\diagcoupon(0.25,1.1)(2.75,1.7){\diff^{\rtriv}}
				\rput[B](2.5,-1.7ex){$\scriptstyle 0$}
				\rput[B](1.0,-1.7ex){$\scriptstyle p$}
		\end{centerpict}
		+ (-1)^p
		\begin{centerpict}[midline=1](0,-2ex)(3,2)
				\psline(0.75,1.5)(0.75,2)
				\psline(1.25,1.5)(1.25,2)
				\psline(1.75,1.5)(1.75,2)
				\psline(2.25,1.5)(2.25,2)
				\pscustom[linestyle=none,fillcolor=bandcolor,fillstyle=solid]{%
					\moveto(0.5,0.0)\curveto(0.8,0.4)(1.0,0.4)(1.0,0.8)
					\lineto(1.5,0.8)\curveto(1.5,0.4)(1.3,0.4)(1.0,0.0)
					\lineto(0.5,0.0)
				}
				\psbezier(0.5,0)(0.8,0.4)(1.0,0.4)(1.0,0.8)\psline(1.0,0.8)(1.0,1.5)
				\psbezier(1.0,0)(1.3,0.4)(1.5,0.4)(1.5,0.8)\psline(1.5,0.8)(1.5,1.5)
				\psbezier(1.5,0)(1.8,0.4)(2.0,0.4)(2.0,0.8)\psline(2.0,0.8)(2.0,1.5)
				\psbezier(2.0,0)(2.3,0.4)(2.5,0.4)(2.5,0.8)\psline(2.5,0.8)(2.5,1.5)
				\psbezier(2.5,0)(2.5,0.2)(0.5,0.1)(0.5,0.5)\psline(0.5,0.5)(0.5,1.5)
				\diagcoupon(0.75,0.4)(2.75,1.0){u^{p-1}}
				\diagcoupon(0.25,1.2)(2.75,1.8){\diff^{\rtriv}}
				\rput[B](2.5,-1.7ex){$\scriptstyle 0$}
				\rput[B](1.0,-1.7ex){$\scriptstyle p$}
		\end{centerpict}
		\\[2ex]
		&=
		\begin{centerpict}[midline=1](0,-2ex)(3,2)
				\bandline(0.75,0)(0.75,0.5)
				\psline(0.5,0)(0.5,0.8)\psline(0.75,0.8)(0.75,2)
				\psline(1.0,0)(1.0,0.8)\psline(1.25,0.8)(1.25,2)
				\psline(1.5,0)(1.5,0.8)\psline(1.75,0.8)(1.75,2)
				\psline(2.0,0)(2.0,0.8)
				\psline(2.5,0)(2.5,2.0)
				\diagcoupon(0.25,0.3)(2.25,0.9){\diff^{\opn}-\diff^{\rtriv}}
				\diagcoupon(0.50,1.1)(2.00,1.7){u^{p-2}}
				\rput[B](2.5,-1.7ex){$\scriptstyle 0$}
				\rput[B](1.0,-1.7ex){$\scriptstyle p$}
		\end{centerpict}
		-(-1)^p
		\begin{centerpict}[midline=1](0,-2ex)(3,2)
				\bandline(0.75,0)(0.75,1)
				\psline(0.5,0)(0.5,2)
				\psline(1.0,0)(1.0,2)
				\psline(1.5,0)(1.5,2)
				\psline(2.0,0)(2.0,2)
				\psline(2.5,0)(2.5,1.5)\psdot(2.5,1.5)
				\diagcoupon(0.25,0.7)(2.25,1.3){u^{p-1}}
				\rput[B](2.5,-1.7ex){$\scriptstyle 0$}
				\rput[B](1.0,-1.7ex){$\scriptstyle p$}
		\end{centerpict}
		+(-1)^p
		\begin{centerpict}[midline=1](0,-2ex)(3,2)
				\pscustom[linestyle=none,fillcolor=bandcolor,fillstyle=solid]{%
					\moveto(0.5,0.0)\curveto(0.8,0.5)(1.0,0.5)(1.0,1.0)
					\lineto(1.5,1.0)\curveto(1.5,0.5)(1.3,0.5)(1.0,0.0)
					\lineto(0.5,0.0)
				}
				\psbezier(0.5,0)(0.8,0.5)(1.0,0.5)(1.0,1.0)\psline(1.0,1.0)(1.0,2)
				\psbezier(1.0,0)(1.3,0.5)(1.5,0.5)(1.5,1.0)\psline(1.5,1.0)(1.5,2)
				\psbezier(1.5,0)(1.8,0.5)(2.0,0.5)(2.0,1.0)\psline(2.0,1.0)(2.0,2)
				\psbezier(2.0,0)(2.3,0.5)(2.5,0.5)(2.5,1.0)\psline(2.5,1.0)(2.5,2)
				\psbezier(2.5,0)(2.5,0.4)(0.5,0.2)(0.5,0.7)\psline(0.5,0.7)(0.5,1.5)\psdot(0.5,1.5)
				\diagcoupon(0.75,0.7)(2.75,1.3){u^{p-1}}
				\rput[B](2.5,-1.7ex){$\scriptstyle 0$}
				\rput[B](1.0,-1.7ex){$\scriptstyle p$}
		\end{centerpict}
		-(-1)^p
		\begin{centerpict}[midline=1](0,-2ex)(3,2)
				\psline(1.25,0.9)(1.25,2)
				\psline(1.75,0.9)(1.75,2)
				\psline(2.25,0.9)(2.25,2)
				\pscustom[linestyle=none,fillcolor=bandcolor,fillstyle=solid]{%
					\moveto(0.5,0.0)\curveto(0.8,0.4)(1.0,0.4)(1.0,0.8)
					\lineto(1.5,0.8)\curveto(1.5,0.4)(1.3,0.4)(1.0,0.0)
					\lineto(0.5,0.0)
				}
				\psbezier(0.5,0)(0.8,0.4)(1.0,0.4)(1.0,0.8)
				\psbezier(1.0,0)(1.3,0.4)(1.5,0.4)(1.5,0.8)
				\psbezier(1.5,0)(1.8,0.4)(2.0,0.4)(2.0,0.8)
				\psbezier(2.0,0)(2.3,0.4)(2.5,0.4)(2.5,0.8)
				\psbezier(2.5,0)(2.5,0.2)(0.5,0.1)(0.5,0.5)\psline(0.5,0.5)(0.5,2)
				\diagcoupon(0.75,0.4)(2.75,1.0){\diff^{\rtriv}-\diff^{\opn}}
				\diagcoupon(1.00,1.2)(2.50,1.8){u^{p-2}}
				\rput[B](2.5,-1.7ex){$\scriptstyle 0$}
				\rput[B](1.0,-1.7ex){$\scriptstyle p$}
		\end{centerpict}
		\\[2ex]
		&=
		\begin{centerpict}[midline=1](0,-2ex)(3,2)
				\bandline(0.75,0)(0.75,0.5)
				\psline(0.5,0)(0.5,0.8)\psline(0.75,0.8)(0.75,2)
				\psline(1.0,0)(1.0,0.8)\psline(1.25,0.8)(1.25,2)
				\psline(1.5,0)(1.5,0.8)\psline(1.75,0.8)(1.75,2)
				\psline(2.0,0)(2.0,0.8)
				\psline(2.5,0)(2.5,2.0)
				\diagcoupon(0.25,0.3)(2.25,0.9){\diff^{\rtriv}-\diff^{\opn}}
				\diagcoupon(0.50,1.1)(2.00,1.7){u^{p-2}}
				\rput[B](2.5,-1.7ex){$\scriptstyle 0$}
				\rput[B](1.0,-1.7ex){$\scriptstyle p$}
		\end{centerpict}
		+(-1)^{p-1}
		\begin{centerpict}[midline=1](0,-2ex)(3,2)
				\bandline(0.75,0)(0.75,0.5)
				\psline(0.5,0)(0.5,0.8)
				\psline(1.0,0)(1.0,0.8)
				\psline(1.5,0)(1.5,0.8)
				\psline(2.0,0)(2.0,0.8)
				\psline(1.25,1.5)(1.25,2)
				\psline(1.75,1.5)(1.75,2)
				\psline(2.25,1.5)(2.25,2)
				\psbezier(0.75,0.6)(0.8,1)(1.25,1)(1.25,1.4)
				\psbezier(1.25,0.6)(1.3,1)(1.75,1)(1.75,1.4)
				\psbezier(1.75,0.6)(1.8,1)(2.25,1)(2.25,1.4)
				\pscustom{%
					\moveto(2.5,0)\lineto(2.5,0.7)
					\curveto(2.5,1.1)(0.75,0.9)(0.75,1.3)\lineto(0.75,2)
				}
				\diagcoupon(0.25,0.2)(2.25,0.8){\diff^{\rtriv}-\diff^{\opn}}
				\diagcoupon(1.00,1.2)(2.50,1.8){u^{p-2}}
				\rput[B](2.5,-1.7ex){$\scriptstyle 0$}
				\rput[B](1.0,-1.7ex){$\scriptstyle p$}
		\end{centerpict}
		+(-1)^{p+1}
		\begin{centerpict}[midline=1](0,-2ex)(3,2)
				\bandline(0.75,0)(0.75,0.5)
				\psline(0.5,0)(0.5,0.8)\psline(0.75,0.8)(0.75,2)
				\psline(1.0,0)(1.0,0.8)\psline(1.25,0.8)(1.25,2)
				\psline(1.5,0)(1.5,0.8)\psline(1.75,0.8)(1.75,2)
				\psline(2.0,0)(2.0,0.8)\psline(2.25,0.8)(2.25,2)
				\psline(2.5,0)(2.5,0.8)
				\diagcoupon(0.25,0.3)(2.75,0.9){d_0^{\rtriv}-d_0^{\opn}}
				\diagcoupon(0.50,1.1)(2.50,1.7){u^{p-1}}
				\rput[B](2.5,-1.7ex){$\scriptstyle 0$}
				\rput[B](1.0,-1.7ex){$\scriptstyle p$}
		\end{centerpict}
		= u^{p-1}(\diff^{\rtriv}-\diff^{\opn}),
	\end{align*}
	and together with $\diff^Rs\alpha = s\alpha\bar\diff^R$ it implies\footnote{
		For clarity we used here $\bar\diff^R$ for the~rack differential in the~normalized complex, to distinguish it from the~unnormalized one.
	}
	\begin{displaymath}
		\diff^R\bisom^p(a\otimes b)
			= \bisom^p(\diff^Ra\otimes b)
			+ (-1)^q \bisom^p(a\otimes\bar\diff^Rb)
	\end{displaymath}
	for $a\in\widehat C_q(M;X)$ and $b\in\CN_p(X)$. The~existence of a~short exact sequence follows from the~K\"unneth theorem.
\end{proof}

\begin{corollary}\label{cor:two-term-HD-from-HN}
	Suppose a~spindle homomorphism $\varphi\colon X\to X'$ induces an~isomorphism on normalized rack homology $\varphi_*\colon \HN(X,\diff^R) \longrightarrow \HN(X',\diff^R)$. If $M$ is an~$X'$-module such that the~induced map $\varphi_*\colon \HN(M^\varphi;X,\diff^R) \longrightarrow \HN(M;X',\diff^R)$ is also an~isomorphism, so is $\varphi_*\colon \HD(M^\varphi;X,\diff^R) \longrightarrow \HD(M;X',\diff^R)$.
\end{corollary}
\begin{proof}
	Use a~similar induction argument to the~one from the~proof of Theorem~\ref{thm:multiterm-HD-from-HN} to show that $B(M^f;X)\cong B(M;X')$.
\end{proof}

\begin{remark}
	The~proof can be easily extended to the~case $\diff^{a,b} = a\diff^{\opn} + b\diff^{\rtriv}$, when one replaces $B(M;X)$ with the~bicomplex $\widehat C(M;X,\diff^{a,b})[2]\otimes\CN(X,b{\cdot}\diff^R)$. In particular, $\HD(X,\diff^{a,b})$ is fully determined by the~rack-type homology $\HN(X,b{\cdot}\diff^R)$, which can be computed from $\HN(X,\diff^R)$; see Lemma 5.3 from \cite{PrzPut-lattices} for a~more detailed statement.
\end{remark}

\appendix
\section{Sectral sequences}\label{sec:spectral}
This section provides basic definitions and results on spectral sequences. Most theorems are left without proves, which can be found for example in \cite{Spectral-sequences}. The~exception is Lemma~\ref{lem:part-isom-on-E2-gives-isom-on-Eoo}, which seems to be unknown. It is a~key ingredient to the~proof of Theorem~\ref{thm:multiterm-HD-from-HN}.

\begin{definition}
	A~\emph{(homological) spectral sequence} is a~sequence of bigraded $R$-modules $\{E^r\}_{r\in \mathbb N}$ together with differentials $d^r\colon E^r \longrightarrow E^r$, each of degree $(-r,r-1)$, such that $H_*(E^r,d^r)\cong E^{r+1}$. The~chain complex $(E^r,d^r)$ is called the~\emph{$r$-th page} of the~spectral sequence $(E,d)$.
\end{definition}

\begin{definition}
	A~\emph{morphism of spectral sequences} $f\colon E \longrightarrow \bar E$ is a~sequence of module homomorphisms $f^r\colon E^r \longrightarrow \bar E^r$ such that $f^{r+1}$ is equal to the~induced homomorphism $f^r_*\colon H(E^r,d^r) \longrightarrow H(\bar E^r, d^r)$.
\end{definition}

We shall be interested only in the~so called \emph{first quadrant} spectral sequences, i.e.\ those with $E^r_{pq}=0$ if $p<0$ or $q<0$. Then it must be $E^r_{pq} = E^{r+1}_{pq}$ for $r>p+q+1$, which means the~sequence \emph{converges}; we shall write $E^{\infty}_{pq}$ for the~limit groups. A~natural source of such spectral sequences is provided by filtered chain complexes.

\begin{definition}
	A~\emph{filtration} of a~graded $R$-module $M$ is a~sequence of graded submodules $\F^0M \subset \F^1M \subset \F^2M \subset \cdots$ such that $\bigcup_{i\in\mathbb N} \F^iM = M$. We say the~filtration is \emph{bounded} if for every $n\in\mathbb Z$ there exists $i\in\mathbb N$ such that $M_n=\F^iM_n$.
\end{definition}

Given a~filtered graded $R$-module $(M,\F)$ we define its \emph{graded associated module} as
\begin{equation}
	\gr M:=\bigoplus_{i\in\mathbb N}\F^{i+1}M / \F^iM.
\end{equation}

\begin{definition}
A~spectral sequence $\{E^r,d^r\}_{r\in\mathbb N}$ \emph{converges} to a~filtered module $M$ if it stabilizes and $E^{\infty}_{pq} \cong \gr_p M_q$.
\end{definition}

Given a~filtered chain complex $(C,\F)$ we filter its homology by images of homology of the~subcomplexes: $\F^pH(C) := \mathrm{im}(H(\F^pC)\to H(C))$. Clearly, $\F H$ is bounded if so is $\F C$.

\begin{theorem}[cf. \cite{Spectral-sequences}]
	Let $(C,\F)$ be a~chain complex with a~bounded filtration. Then there exists a~spectral sequence $(E^r,d^r)$ with $E^1_{pq} = H_{p+q}(\F^pC/\F^{p-1}C)$ converging to $H(C)$.
\end{theorem}

Given filtered chain complexes $(C,\F)$ and $(C',\F$) we say a~chain map $f\colon C \longrightarrow C'$ is \emph{filtered} if $f(\F^pC)\subset \F^pC'$ for any $p$. Such a~map induces a~morphism on the~associated quotients $f_*\colon \F^pC/\F^{p-1}C \longrightarrow \F^pC'/\F^{p-1}C'$ and a~homomorphism of spectral sequences $f^r\colon E^r \longrightarrow \bar E^r$ (the~associated quotients form the~$0$-th page).

\begin{theorem}[cf. \cite{Spectral-sequences}]
	Assume $f\colon (C,\F) \longrightarrow (C',\F)$ is a~filtered chain map and the~filtrations on $C$ and $C'$ are bounded. Denote by $E^r$ and $\bar E^r$ the~spectral sequences for $(C,\F)$ and $(C',\F)$ respectively. If $f^{\infty}_{pq}\colon E^{\infty}_{pq} \longrightarrow \bar E^{\infty}_{pq}$ is an~isomorphism for all $p+q=n$, then so is $f_*\colon H_n(C) \longrightarrow H_n(C')$.
\end{theorem}

In other words, if $f$ induces an~isomorphism on graded modules associated to homology, then the~map of homology is also an~isomorphism. In fact, it is enough to have an~isomorphism at certain page to deduce that homology of $C$ and $C'$ are isomorphic.

\begin{corollary}
	In the~situation as above assume that $f^r\colon E^r \longrightarrow \bar E^r$ is an~isomorphism for some $r\in\mathbb N$. Then $f_*\colon H(C) \longrightarrow H(C')$ is an~isomorphism.
\end{corollary}

\noindent
However, we need a~weaker result with $f^r$ being an~isomorphism only in certain degrees.

\begin{lemma}\label{lem:part-isom-on-E2-gives-isom-on-Eoo}
	Assume there is a~homomorphism of spectral sequences $f^r_{pq} \colon E^r_{pq} \longrightarrow \bar E^r_{pq}$ such that $f^2_{pq}$ is an~isomorphism for $q \leqslant N$. Then $f^{\infty}_{pq}\colon E^{\infty}_{pq} \longrightarrow \bar E^{\infty}_{pq}$ is an~isomorphism for $p+q=N$.
\end{lemma}

The~main idea is that if a~chain map $g_i\colon C_i \longrightarrow C'_i$ between two chain complexes is an~isomorphism for $i = i_0,i_0-1$, and $i_0+1$, then $g_*\colon H_i(C) \longrightarrow H_i(C')$ is an~isomorphism. Indeed, let $t_iC$ be the~subquotient of $C$ with $(t_iC)_j=C_j$ whevener $j\in\{i-1,i,i+1\}$ and zero otherwise. Likewise we define $t_iC'$. Then $H_i(t_iC)\cong H_i(C)$ and similarly for $C'$, whereas $g\colon C\longrightarrow C'$ descends to an~isomorphism between $t_iC$ and $t_iC'$.

Therefore, as $f^2_{pq}$ is an~isomorphism for $q \leqslant N$, and the~differential $d^2$ increases $q$ by 1 and decreases $p$ by 2, $f^3_{pq}$ is an~isomorphism whenever $q \leqslant N-1$ or $q=N$ and $p<2$. The~next differential, $d^3$, increases $q$ by 2 and decreases $p$ by 3. Hence, $f^4_{pq}$ is an~isomorphism for $q \leqslant N-3$ and a~few more points: $p<2$ and $q\leqslant N$, or $p=2$ and $q\leqslant N-1$, or $p=3,4$ and $\leqslant N-2$. By repeating this process we see, that the~set of points $(p,q)$ for which $f^r_{pq}$ is an~isomorphism form a~staircase diagram, see Fig.~\ref{fig:staircase}. We want to show that the~whole diagonal $p+q=N$ is below the~stairs for any $r$.

\begin{figure}
	\newgray{stepOne}{0.9}
	\newgray{stepTwo}{0.8}
	\newgray{stepThree}{0.7}
	\newgray{stepFour}{0.6}
	\newgray{stepFive}{0.5}
	\begin{pspicture}(-1,1)(9,-4.5)
		\psset{unit=5mm}%
		\begingroup
			\psset{linestyle=none,fillstyle=solid}%
			\psframe[fillcolor=stepFive](0,0)(16,-8)
			\pspolygon[fillcolor=stepOne](16, 0)(1, 0)(2,-1)(16,-1)
			\pspolygon[fillcolor=stepTwo](16,-1)(2,-1)(3,-2)(4,-2)(5,-3)(16,-3)
			\pspolygon[fillcolor=stepThree](3,-2)(4,-3)(5,-3)(4,-2)
			\pspolygon[fillcolor=stepThree](16,-3)(5,-3)(7,-5)(8,-5)(9,-6)(16,-6)
			\pspolygon[fillcolor=stepFour](4,-3)(5,-4)(6,-4)(5,-3)
			\pspolygon[fillcolor=stepFour](16,-6)(9,-6)(8,-5)(7,-5)(9,-7)(10,-7)(11,-8)(16,-8)
		\endgroup
		\multips(0,0)(0,-1){9}{\multips(0,0)(1,0){17}{\psdot[dotsize=2pt](0,0)}}
		\psline{<->}(0,2)(0,-8)(17,-8)
		\psline(0,0)(8,-8)
		\uput[l](0,2){$\scriptstyle q$}
		\uput[d](17,-8){$\scriptstyle p$}
		\uput[l](0, 0){$\scriptstyle N$}\psline(-2pt, 0)(2pt, 0)
		\uput[d](8,-8){$\scriptstyle N$}\rput(8,-8){\psline(0,-2pt)(0,2pt)}
		\uput[r](16,-0.5){$\scriptstyle r\leqslant 2$}
		\uput[r](16,-2.0){$\scriptstyle r\leqslant 3$}
		\uput[r](16,-4.5){$\scriptstyle r\leqslant 4$}
		\uput[r](16,-7.0){$\scriptstyle r\leqslant 5$}
		\rput(2.5,-5.5){$\scriptstyle r\geqslant 6$}
	\end{pspicture}
	\caption{The~sets $\mathcal I^r$ of points $(p,q)$ at which $f^r_{pq}$ is an isomorphism, visualized for $N=8$. A~region labeled $r\leqslant i$ belongs to every $\mathcal I^r$ with $r\leqslant i$. The~skew line represents the~diagonal $p+q=N$}%
	\label{fig:staircase}
\end{figure}
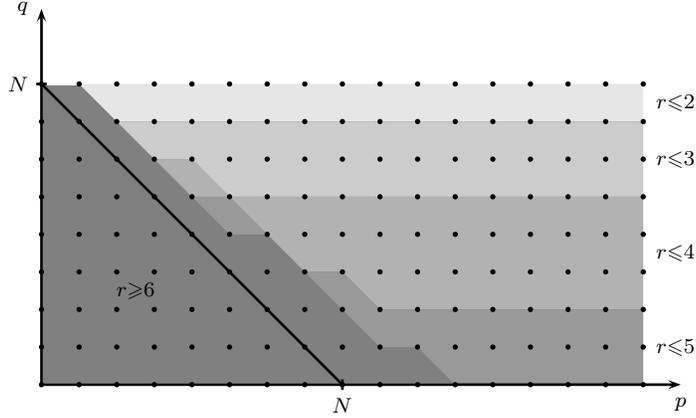

\begin{proof}[Proof of Lemma~\ref{lem:part-isom-on-E2-gives-isom-on-Eoo}]
	Define functions $p^i(r) := \sum_{k=1}^i(r-k)$. In particular, we have
	\begin{equation}
		p^i(r+1)=p^i(r)+i \hskip 2cm
		p^{i+1}(r+1)=p^i(r)+r.
	\end{equation}
	Let $\mathcal I^r \subset \mathbb{N} \times \mathbb{N}$ consists of pairs $(p,q)$ satisfying the~following condition:
	\begin{equation}\label{eq:cond-for-(p,q)}
		\begin{cases}
			q \leqslant N,   & \textrm{ if } p=0,\\
			q \leqslant N+i-p, & \textrm{ if } p^{i-1}(r) < p \leqslant p^i(r) \textrm{ for some } i,\\
			q \leqslant N+(r-2)-p^{r-2}(r), & \textrm{ if } p > p^{r-2}(r).
		\end{cases}
	\end{equation}
	Notice that all $\mathcal{I}^r$ contain the~diagonal $p+q=N$, and $f^2_{pq}$ is an~isomorphism for $(p,q)\in\mathcal I^2$. We shall show by induction that $f^{r+1}_{pq}$ is also an~isomorphism whenever $(p,q)\in\mathcal I^{r+1}$. For that we have to check that $f^r_{pq}$, $f^r_{p+r,q-r+1}$, and $f^r_{p-r,q+r-1}$ are isomorphisms, as $d^r$ lowers $p$ by $r$ and increases $q$ by $r-1$.
	\begin{itemize}
		\item Since $p^i(r)<p^i(r+1)$, we have $\mathcal I^{r+1} \subset \mathcal I^r$ and $(p,q)\in\mathcal I^r$. Hence, $f^r_{pq}$ is an~isomorphism.
		
		\item Clearly $f^r_{p+r,q-r+1}$ is an~isomorphism if $q<r-1$ or $p>p^{r-1}(r+1)$ (the~last case in \ref{eq:cond-for-(p,q)}). Assume then that $p^i(r+1)\leqslant p < p^{i+1}(r+1)$ for some $i$. Then $p+r > p^{i+1}(r+1)=p^i(r)+r>p^{i+1}(r)$, and $(p+r)+(q-r+1)=p+q+1\leqslant N+i+1$, so that $(p+r,q-r+1)\in\mathcal I^r$.
		
		\item Finally, $f^r_{p-r,q+r-1}$ is an~isomorphism if $p<r$, so that we can assume $p \geqslant p^i(r+1)$ for some $i>0$ (as $p^1(r+1)=r$). There are two subcases to check.
		\begin{itemize}
			\item \emph{Case 1:} $p>p^{r-1}(r+1)$. We have
			\begin{displaymath}
				q+r-1 \leqslant N+2(r-1)-p^{r-1}(r+1) = N+(r-2)-p^{r-2}(r),
			\end{displaymath}
			which guarantees $(p-r,q+r-1)\in\mathcal I^r$.
			
			\item \emph{Case 2:} $p^i(r+1) < p\leqslant p^{i+1}(r+1)$ for some $i=0,\dots,r-1$ (case $i=0$ happens if $p=r$). Then $p^{i-1}(r) < p-r\leqslant p^i(r)$ and since
			\begin{displaymath}
				(p-r)+(q+r-1)=p+q-1\leqslant N+i,
			\end{displaymath}
			we have $(p-r,q+r-1)\in\mathcal I^r$.
		\end{itemize}
	\end{itemize}
	Hence, $f^{r+1}_{pq}$ is an~isomorphism if $(p,g)\in\mathcal I^r$, and so is $f^{\infty}_{pq}$ for $p+q\leqslant N$.
\end{proof}

\setsingleline

\end{document}